%% file: main.tex
\documentclass[12pt]{article}
\usepackage[utf8]{inputenc}
\usepackage{geometry}
\usepackage{url}
\usepackage{tikz}
\usepackage{amsmath,amssymb,amsthm}
\usepackage{subcaption}
\usepackage{mleftright}
\mleftright
\usepackage{color}
\usepackage{numprint}
\usepackage{booktabs}
\usepackage{makecell}
\usepackage[section]{placeins}
\newtheorem{theorem}{Theorem}

\newtheorem{lemma}[theorem]{Lemma}

\newtheorem{observation}[theorem]{Observation}
\newtheorem{conjecture}[theorem]{Conjecture}
\newtheorem{problem}[theorem]{Problem}
\newtheorem{claim}{Claim}
\theoremstyle{definition}                    

\theoremstyle{remark}

\title{The lengths for which bicrucial square-free permutations exist}
\author{Carla Groenland\thanks{Utrecht University, the Netherlands. Partially supported by the European Research Council Horizon 2020 project CRACKNP (grant agreement no. 853234). \newline E-mail: \texttt{c.e.groenland@uu.nl}}~ and Tom Johnston\thanks{Mathematical Institute, University of Oxford, Oxford, OX2 6GG, United Kingdom.\newline  E-mail: \texttt{thomas.johnston@maths.ox.ac.uk}.}}
\date{\today}
\newcommand{\oiso}{\sim}
\tikzstyle{vertex}=[circle, draw, fill=black, inner sep=0pt, minimum width=4pt]

\newcommand{\ceil}[1]{\left\lceil #1 \right\rceil}

\begin{document}
\maketitle

\begin{abstract}
    A square is a factor $S = (S_1; S_2)$ where $S_1$ and $S_2$ have the same pattern, and a permutation is said to be square-free if it contains no non-trivial squares. The permutation is further said to be bicrucial if every extension to the left or right contains a square. We completely classify for which $n$ there exists a bicrucial square-free permutation of length $n$.
\end{abstract}
\section{Introduction}
A word $W$ contains a \emph{square} if it has two equal consecutive factors, that is, if $W = U X X V$ for some non-empty word $X$. For example, the word 102021 contains the square 0202 while the word 01213 is square-free. Squares in words have been well studied since Thue \cite{thue1906uber} proved that there is an infinite square-free word over an alphabet of size 3, while it is easy to see that there are no square-free words of length at least 4 over an alphabet of size 2. Thue also showed that there exist infinitely many cube-free words over a 2-letter alphabet.

A word $W$ is \emph{left-crucial} (respectively \emph{right-crucial}) with respect to squares if it does not itself contain any square, but any word obtained by prepending (respectively appending) a single letter contains a square. We say $W$ is \emph{bicrucial} (also known as \emph{maximal} in the literature) if it is both left-crucial and right-crucial. Right-crucial words were characterised by Li \cite{li1976annihilators} in 1976, and from this characterisation it follows that, up to permuting the alphabet, the smallest bicrucial word on an alphabet of size $n$ is the $n$th \emph{Zimin word} (where Zimin words are defined recursively as $W_1 =  1$ and $W_n = W_{n-1}nW_{n-1}$). It is also known that any square-free word is a factor of a bicrucial square-free word \cite{bean1979avoidable}, and the result of Thue \cite{thue1906uber} shows there  are infinitely many square-free words. These two results combine to show that there are arbitrarily long bicrucial square-free words. 

In 2011, Avgustinovich, Kitaev, Pyatkin and Valyuzhenich~\cite{avgustinovich2011square} initiated the study of bicrucial permutations, showing that such permutations exist of all odd lengths $8k+1$, $8k+5$ and $8k+7$ for $k \geq 1$. 
Using the constraint solver Minion, bicrucial permutations of length $8k+3$ for $k=2,3$ and of even length $32$ were found by Gent, Kitaev, Konovalov, Linton and Nightingale~\cite{gent2015crucial}, and they conjectured that larger permutations should exist.

\begin{conjecture}[Conjecture 8 in \cite{gent2015crucial}]
There exist bicrucial permutations of length $8k + 3$ for all $k \geq 2$.
\end{conjecture}

\begin{conjecture}[Conjecture 9 in \cite{gent2015crucial}]
There exist arbitrarily long bicrucial permutations of even length.
\end{conjecture}

We confirm both these conjectures and completely classify the $n$ for which there exist bicrucial permutations of length $n$.
\begin{theorem}
\label{thm:main}
A bicrucial permutation of length $n$ exists if and only if $n=9$, $n\geq 13$ is odd  or $n\geq 32$ is even and not $38$.
\end{theorem}
The proof of this result is split into several parts: In Section \ref{sec:large-n} we give a construction of bicrucial permutations of length $8k + 3$ for every $k \geq 3$, which completes the classification of the odd $n$, and a construction for when $n \geq 48$ is even. The remaining small (even) cases are handled using computer searches, and we give constructions for the remaining possible $n$ in Section \ref{sec:small-n}. To show that there are no bicrucial permutations of length 38 we use an exhaustive computer search that we detail in Section \ref{sec:no-38}. The code used to exhaustively search for such permutations can also be used to enumerate square-free, left-crucial and bicrucial permutations with only minor modifications, and we do this in Section \ref{sec:enumeration}.

\section{Preliminaries}
\label{sec:prelim}
Let $\sigma \in S_n$ be a permutation of $\{0, \dots, n -1\}$ of length $n$. As we construct our bicrucial permutations it will be convenient write $\sigma$ as a vector where the $i$th entry is $\sigma(i)$ and where we index our vectors starting from 0. That is \[\sigma = \left( \sigma(0) , \sigma(1), \dots, \sigma(n - 1) \right) = \left(\sigma_0, \sigma_1, \dots \sigma_{n-1}\right).\]

Given a vector $a = (a_1, \dots, a_{n})$ consisting of $n$ distinct real numbers, one can associate a unique permutation $\sigma \in S_n$ (written as a vector) by replacing the $i$th smallest entry of $a$ with $i$, and we call this permutation the \emph{pattern} of $a$. We say two vectors $a$ and $b$ are \emph{order-isomorphic} if they have the same pattern, and in this case we write $a \sim b$. For example, the vector $(5, 2, 4, 10)$ has the pattern $(2, 0, 1, 3)$.

A \emph{square} of length $2\ell$ (for $\ell\geq 2$) is a factor
\[
(S_1; S_2) = (\sigma_k,\dots,\sigma_{k+\ell-1}; \sigma_{k+\ell},\dots,\sigma_{k+2\ell-1})
\]
where $S_1 \oiso S_2$, and we say a permutation is \emph{square-free} if it contains no squares of length at least $4$. For each $i\in \{0,\dots,\ell-1\}$, the entry $\sigma_{k+i}$ has the same position in the pattern as the entry $\sigma_{k + \ell +i}$ and we will say that the entry $\sigma_{k+i}$ \emph{corresponds} to the entry $\sigma_{k + \ell +i}$.

A permutation $\pi$ of length $n + 1$ is a \emph{right-extension} of $\sigma$ if $(\pi_0, \dots, \pi_{n-1}) \oiso \sigma$, or equivalently, if $\pi$ can be formed by appending an entry $x \in \{0, \dots, n\}$ to $\sigma$ and replacing $\sigma_i$ by $\sigma_i + 1$ if $\sigma_i \geq x $. We say that the permutation $\sigma$ is \emph{right-crucial} if it is square-free, and every right-extension $\pi$ contains a square, and we define left-extensions and left-crucial permutations similarly. We will be interested in permutations which are simultaneously left-crucial and right-crucial which we call \emph{bicrucial} permutations.

\paragraph{Up-up-down-down condition}
For $i\in \{0,\dots,n-2\}$, we say a permutation $\sigma\in S_n$ goes \emph{up} from position $i$ to position $i+1$ if $\sigma_i<\sigma_{i+1}$ and goes \emph{down} otherwise. The permutation $\sigma$ satisfies the \emph{up-up-down-down condition} if for every $i\in\{0,\dots,n-4\}$, $\sigma$ goes up from position $i$ to position $i+1$ if and only if it goes down from $i+2$ to $i+3$. This forces the permutation to alternate between going up twice and going down twice. If a permutation $\sigma$ violates the up-up-down-down condition and goes up from position $i$ to $i+1$ and also up from position $i+2$ to $i+3$ say, then $(\sigma_{i}, \sigma_{i+1}; \sigma_{i+2}, \sigma_{i+3})$ is a square and the permutation is not square-free.  In fact, the up-up-down-down condition is equivalent to $\sigma$ not containing any squares of length $4$.

In any permutation which satisfies the up-up-down-down condition, the entries fall naturally into three categories. We say the entry $i$ is a \emph{high entry} if $\sigma$ goes up from $i-1$ to $i$ and down from $i$ to $i+1$, a \emph{low entry} if $\sigma$ goes down from $i-1$ to $i$ and up from $i$ to $i+1$ and a \emph{medium entry} otherwise. This definition extends to the initial and final entries in the obvious way.

\paragraph{High-medium-low construction}
We now give a construction due to  Avgustinovich, Kitaev, Pyatkin and Valyuzhenich \cite{avgustinovich2011square} which, starting from a square-free permutation $\pi$ of length $n$, constructs a square free permutation $\sigma$ of length $2n$.  Let $\sigma$ be any permutation such that $\sigma_i < \ceil{n/2}$ when $i \equiv 1 \bmod 4$, $\sigma_i = \pi_i + \ceil{n/2}$ when $i$ is even and $\sigma_i \geq n + \ceil{n/2}$ when $i \equiv 3 \bmod 4$.  It is not hard to see that this permutation is square-free. Indeed, the permutation satisfies the up-up-down-down condition, so there are no squares of length 4 and any square must be of length at least 8. Restricting a square in $\sigma$ to only the terms of even index (in $\sigma$) gives a square of length at least 4 in $\pi$. Since $\pi$ is square-free, $\sigma$ must also be square-free.

\medskip

This simple high-medium-low construction is very useful. By suitably choosing the pattern of the prefix and the suffix of our permutation, we can ensure the permutation has a square when extending to the left and to the right. The difficulty is then in joining the prefix and suffix with a long middle section while ensuring that the permutation is square-free. To do this we construct the middle section using the high-medium-low construction above with additional constraints on $\pi$.
The key to doing this is the observation that every high entry is higher than every medium entry, which are all in turn higher than every low entry, while this not going to be true in the chosen prefix/suffix.

To prove that our constructions give square-free permutations, we will use the following simple observation. 
\begin{observation}
Suppose $\sigma_i$ is contained in a square and the corresponding entry in the other half of the square is $\sigma_j$. Then $|i - j | \leq n /2$.
\end{observation}
In particular, entries in our short prefix can only correspond to entries at most a little over halfway through the permutation and, provided $n$ is large enough, cannot correspond to entries in the suffix. This means we only need to avoid our prefix corresponding to part of the high-medium-low construction. When $n$ is small, we will check the construction using a computer.

\section{Constructions for large $n$}
\label{sec:large-n}
We first show the existence of bicrucial permutations when $n$ is of the form $8k+3$ and $k \geq 3$. Since constructions are already known when $n$ is of the form $8k +1$, $8k+5$ or $8k+7$ ($k \geq 1$) \cite{avgustinovich2011square} and the number of bicrucial permutations is known when $n \leq 19$ (OEIS  A238935 \cite{oeis}), this completely classifies the odd $n$ for which bicrucial permutations exist. 

\begin{figure}
    \centering
    \begin{tikzpicture}[xscale=\textwidth/52cm, yscale=8cm/83cm]
    
\draw node[style=vertex](0) at (0,0) {};
\draw node[style=vertex](1) at (1,6) {};
\draw node[style=vertex](2) at (2,5) {};
\draw node[style=vertex](3) at (3,2) {};
\draw node[style=vertex](4) at (4,4) {};
\draw node[style=vertex](5) at (5,7) {};
\draw node[style=vertex](6) at (6,3) {};
\draw node[style=vertex](7) at (7,1) {};
\draw node[style=vertex](8) at (8,8) {};
\draw node[style=vertex](9) at (9,59) {};
\draw node[style=vertex](10) at (10,46) {};
\draw node[style=vertex](11) at (11,18) {};
\draw node[style=vertex](12) at (12,44) {};
\draw node[style=vertex](13) at (13,60) {};
\draw node[style=vertex](14) at (14,35) {};
\draw node[style=vertex](15) at (15,19) {};
\draw node[style=vertex](16) at (16,43) {};
\draw node[style=vertex](17) at (17,61) {};
\draw node[style=vertex](18) at (18,47) {};
\draw node[style=vertex](19) at (19,20) {};
\draw node[style=vertex](20) at (20,39) {};
\draw node[style=vertex](21) at (21,62) {};
\draw node[style=vertex](22) at (22,36) {};
\draw node[style=vertex](23) at (23,21) {};
\draw node[style=vertex](24) at (24,42) {};
\draw node[style=vertex](25) at (25,63) {};
\draw node[style=vertex](26) at (26,48) {};
\draw node[style=vertex](27) at (27,22) {};
\draw node[style=vertex](28) at (28,45) {};
\draw node[style=vertex](29) at (29,64) {};
\draw node[style=vertex](30) at (30,37) {};
\draw node[style=vertex](31) at (31,23) {};
\draw node[style=vertex](32) at (32,41) {};
\draw node[style=vertex](33) at (33,65) {};
\draw node[style=vertex](34) at (34,49) {};
\draw node[style=vertex](35) at (35,24) {};
\draw node[style=vertex](36) at (36,40) {};
\draw node[style=vertex](37) at (37,66) {};
\draw node[style=vertex](38) at (38,38) {};
\draw node[style=vertex](39) at (39,25) {};
\draw node[style=vertex](40) at (40,79) {};
\draw node[style=vertex](41) at (41,84) {};
\draw node[style=vertex](42) at (42,81) {};
\draw node[style=vertex](43) at (43,78) {};
\draw node[style=vertex](44) at (44,80) {};
\draw node[style=vertex](45) at (45,85) {};
\draw node[style=vertex](46) at (46,77) {};
\draw node[style=vertex](47) at (47,76) {};
\draw node[style=vertex](48) at (48,83) {};
\draw node[style=vertex](49) at (49,86) {};
\draw node[style=vertex](50) at (50,82) {};
\draw (0) -- (1) -- (2) -- (3) -- (4) -- (5) -- (6) -- (7) -- (8) -- (9) -- (10) -- (11) -- (12) -- (13) -- (14) -- (15) -- (16) -- (17) -- (18) -- (19) -- (20) -- (21) -- (22) -- (23) -- (24) -- (25) -- (26) -- (27) -- (28) -- (29) -- (30) -- (31) -- (32) -- (33) -- (34) -- (35) -- (36) -- (37) -- (38) -- (39) -- (40) -- (41) -- (42) -- (43) -- (44) -- (45) -- (46) -- (47) -- (48) -- (49) -- (50) ;

\draw[dashed] (-1,13) -- (51,13);
\draw[dashed] (-1, 30) -- (51,30);
\draw[dashed] (-1, 54) -- (51, 54);
\draw[dashed] (-1, 71) -- (51, 71);
\end{tikzpicture}
\caption{A bicrucial permutation of length 51 constructed as in the proof of Theorem \ref{thm:8k+3}. The dashed lines separate the different regions.}
\label{fig:bicrucial51}
\end{figure}
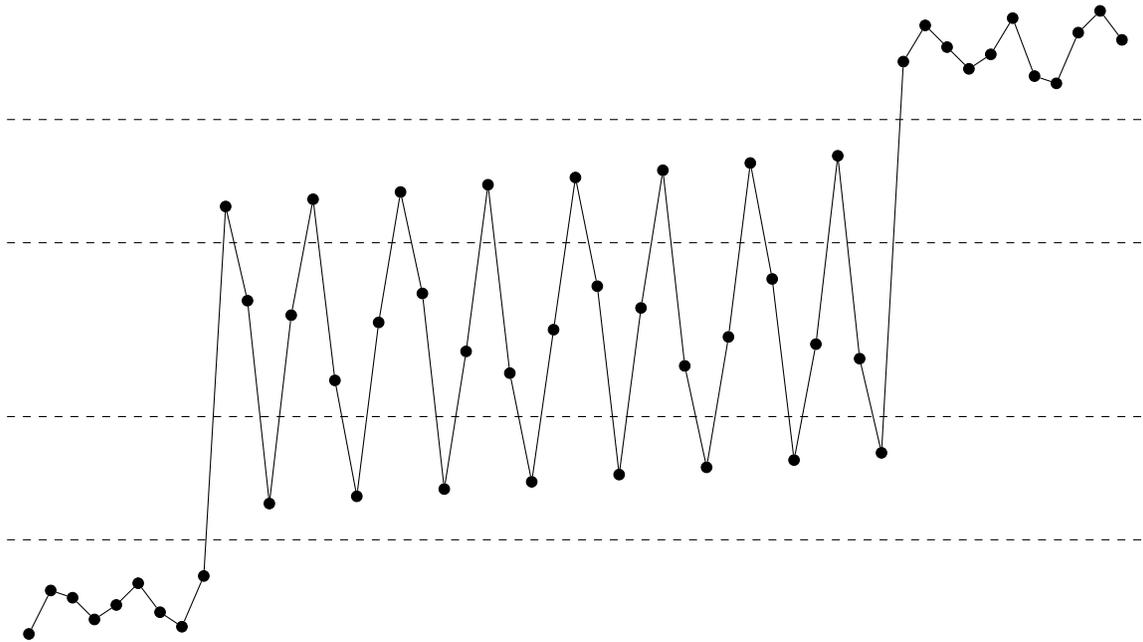

\begin{theorem}
\label{thm:8k+3}
Let $n = 8k + 3$ where $k \geq 3$. Then there exists a bicrucial permutation of length $n$.
\end{theorem}
\begin{proof}
Throughout this proof we shall assume that $n > 44$. This is chosen so that $11 + n/2 < n -11$ and $n - 14 - n/2 > 8$ which will allow us to assume that the relevant corresponding entries come from the middle of the construction. When $n \leq 44$, the proof we give here does not apply. Instead, we have verified using a computer that there exists a bicrucial permutation by using the construction below.\footnote{We repeatedly apply the high-medium-low construction to construct a single permutation $\pi$, and then check that the construction below yields a bicrucial permutation for that particular choice of $\pi$.}.

We split the set $\{0, \dots, n -1\}$ into five different regions $R_0, \dots, R_4$ where $R_i = [r_i, r_{i+1})$ for a specific choice of the values $0 = r_0 \leq r_1 \leq \dotsb \leq r_4 \leq r_5 = n$. The values of the $r_i$ could be determined explicitly from the construction below if desired, but we will leave them only implicitly defined. We begin with a carefully chosen prefix in region $R_0$, followed by a middle section which uses the high-medium-low construction from Section \ref{sec:prelim} across the regions $R_1$, $R_2$ and $R_3$. Finally, we add a carefully chosen suffix in region $R_4$.

Let $m = (n -21)/2$ and let $\pi$ be any square-free permutation of length $m$ which begins down-down (i.e. with the pattern $(2, 1, 0)$). Since every square-free permutation follows the up-up-down-down pattern and trimming the ends of a square-free permutation leaves another square-free permutation, it is trivial to construct $\pi$. We now construct $\sigma$ as follows. Start with the permutation
\(\left( 0, 6, 5, 2, 4, 7, 3, 1, 8 \right),\)
and for $0 \leq i \leq 2m$, define $\sigma_{9+i}$ by
\[\sigma_{9+i} = \begin{cases}
    r_3 + i/4 & i \equiv 0 \mod 4,\\
    r_2 + \pi_{(i-1)/2} & i \equiv 1, 3 \mod 4,\\
    r_1 + (i-2)/4 & i \equiv 2 \mod 4.
\end{cases}\]
It is clear the permutation constructed so far follows the up-up-down-down pattern as required and, since $2m \equiv 2 \bmod 4$, the permutation currently ends down-down.
Finally, append the vector \[\left( r_4+3, r_4+8, r_4+5, r_4+2, r_4+4, r_4+9, r_4+1, r_4+0, r_4+7, r_4+10, r_4+6\right),\]
and note that $\sigma$ still follows the required up-up-down-down pattern.

By our choice of $\pi$, we have \[\left( \sigma_0, \dots, \sigma_{15}\right) \oiso \left(0, 6, 5, 2, 4, 7, 3, 1, 8, 14, 13, 9, 12, 15, 11, 10\right),\]
and this is left-crucial. Indeed, if this is extended to the left by prepending $x \geq 1$, the permutation begins down-up-down and contains a square of length 4. When prepending $0$, the permutation begins with the pattern 
\[\left(0, 1, 7, 6, 3, 5, 8, 4, 2, 9, 15, 14, 10, 13, 16, 12, 11\right),\]
which is a square as \[\left(0, 1, 7, 6, 3, 5, 8, 4\right) \sim \left(0, 1, 6, 5, 2, 4, 7, 3\right) \sim \left(2, 9, 15, 14, 10, 13, 16, 12 \right).\]
By construction, $\sigma$ ends with the pattern $\left( 3, 8, 5, 2, 4, 9, 1, 0, 7, 10, 6\right)$ which is right-crucial. When appending $x \leq 6$, there is a square of length $8$, and when appending $x \geq 7$ there is a square of length 4. Hence, any extension of $\sigma$ to the left or to the right contains a square, and it only remains to check $\sigma$ is square-free. 

Let $S = (S_1; S_2)$ be a square in $\sigma$ and note that $S_1$ has length at least 4. Since the middle section is constructed using the high-medium-low construction, it is square-free and it must be the case that $S_1$ contains one of the first 9 entries or that $S_2$ contains one of the last 11 entries.

Suppose first that $S_1$ contains some of the first 9 entries. If $S_1$ has length 4, then the square is contained entirely in the first 16 entries (the worst case being $\left(\sigma_{8}, \dots, \sigma_{11}; \sigma_{12}, \dots, \sigma_{15}\right)$). However, we know
\[\left( \sigma_0, \dots, \sigma_{15}\right) \oiso \left(0, 6, 5, 2, 4, 7, 3, 1, 8, 14, 13, 9, 12, 15, 11, 10\right),\] and this is square-free. Similarly, $S_1$ cannot be the first 8 entries. Hence, we can assume that $S_1$ contains $\sigma_{8}$ and has length at least 8.

If $S_1$ contains $(\sigma_5, \sigma_6, \sigma_7, \sigma_{8})$, then it contains the pattern $(2, 1, 0, 3)$. The corresponding entries are $(\sigma_{j +5}, \sigma_{j+6}, \sigma_{j+7}, \sigma_{j +8})$ for some $4 \leq j \leq n /2$. Since $n > 44$, we have $n/2 + 8 < n -11$ and the corresponding entries must come from the middle section, but the middle section follows the high-medium-low construction and does not contain any such pattern as the first, high entry is lower than the last, medium entry. If $S_1$ does not contain $(\sigma_5, \sigma_6, \sigma_7, \sigma_{8})$, then it must contain both $\sigma_{8}$ which is a medium entry and $\sigma_{11}$ which is a low entry, but $\sigma_{8} < \sigma_{11}$. The corresponding entries must again be in the middle section by our choice of $n$, but every low entry is lower than any medium entry by construction. This gives a contradiction and $S_1$ does not intersect the first 9 entries.

We now turn our attention to the case where $S_2$ intersects the last 11 entries and apply a very similar argument. Since $m \equiv 3 \bmod 4$ and $\pi$ is a square-free permutation of length $m$ which begins down-down, $\pi$ must end down-down. In particular, the pattern of the last 18 entries in our construction is entirely determined and $\sigma$ ends in the square-free pattern
\[\left(5, 4, 0, 3, 6, 2, 1, 10, 15, 12, 9, 11, 16, 8, 7, 14, 17, 13\right).\]
In particular, the last $11+7$ entries contain no square of length $4$.
This means we can assume that $S_2$ has length at least 8. It cannot be the case that $S_2$ contains the factor $\left(\sigma_{n - 14}, \dots, \sigma_{n-11}\right)$ as this has the pattern $(2, 1, 0, 3)$ and the corresponding entries must be in the middle section where this pattern never occurs. As $S_2$ intersects the last 11 entries, it must start at one of $\sigma_{n-13}, \dots, \sigma_{n-8}$.  Since the medium entry $\sigma_{n - 13}$ is lower than the low entry $\sigma_{n-8}$, $S_2$ cannot contain both $\sigma_{n-13}$ and $\sigma_{n-8}$, which rules out $\sigma_{n-13}$. If $S_2$ were to start at one of $\sigma_{n-12}, \dots, \sigma_{n-8}$, then it would contain the low entry $\sigma_{n-8}$ and the medium entry $\sigma_{n-5}$. However, the corresponding entries from $S_1$ would come from the middle section where all low entries are lower than all medium entries, and $S$ would not be square.
Hence, $S_2$ does not intersect with the last 11 entries and we are done.
\end{proof}

We now turn our attention to the even lengths for which we use a similar but more complicated construction. We will use different length suffixes depending on the parity of $n$ modulo 8. The suffixes are also more complicated and we will have to split into 7 regions to guarantee that the construction is square-free.

\begin{theorem}
\label{thm:even}
There exist bicrucial permutations for all even $n \geq 48$.
\end{theorem}

\input{72}

\begin{proof}
We will assume that $n > 94$. When $n \leq 94$, this proof does not apply, and instead we (again) use a computer to check the construction. Throughout this proof we will claim that certain small permutations are left-/right-crucial and we defer the proofs to Appendix \ref{sec:small-ends}.

We split the space $\{0, \dots, n - 1\}$ into 7 different regions $R_0, \dots, R_6$ where $R_i = [r_i, r_{i+1})$ for some specific choice of the values $0 = r_0 \leq r_1 \leq \dotsb \leq r_6 \leq r_7 = n$ determined implicitly by the construction. 
Let $m =  32 + 2k - 1$ where $k \geq 5$ is equal to $1 \bmod 4$. We will prove that there exist bicrucial permutations of lengths $m + 7, \dots, m + 13$ by extending the same left-crucial permutation $\tau$ of length $m$. Since the permutations all begin and end with fixed patterns which we have chosen to be left-crucial and right-crucial respectively, the only difficulty will be ensuring that the permutations are square-free. First, we need to find an arbitrarily long square-free permutation that satisfies certain conditions on the endpoints. 

\begin{claim}
\label{claim:perm}
There exists a square-free permutation $\pi$ of length $k$ such that $\pi_{k-1} = k-1$ and $\max\{\pi_0, \pi_1, \pi_2\} = \pi_1$.
\end{claim}
\begin{proof}
Let $\pi'$ be a square-free permutation of length $(k-1)/2$. Form the first $k-1$ entries of $\pi$ by using a variation on the high-medium-low construction which starts with a `medium' entry before going to a `high' entry, and then append $k -1$. This is clearly a permutation where $\max\{\pi_0, \pi_1, \pi_2\} = \pi_1$ and $\pi_{k-1} = k-1$, so it remains to check that it is square-free. By construction any possible square must end at $\pi_{k-1}$ and, since $\pi$ follows the up-up-down-down construction, any square must be of length at least $8$. In particular, the square must end in the pattern $(2, 1, 0, 3)$ which is not found anywhere else in the permutation.
\end{proof}
We now choose $\pi$ as in Claim \ref{claim:perm} and construct the permutation $\tau$ of length $m$. Start with the left-crucial permutation (see Lemma \ref{lem:left_crucial}) of $\{0, \dots, 31\}$ given by 
\begin{align*}
    &(8, 2, 0, 20, 29, 19, 3, 4, 7, 5, 1, 30, 31, 21, 6, 9, 22, 12, 10, 24, 26, 23, 13, 14, 17, 15, 11,\\ &\quad \quad 27, 28, 25, 16, 18).
\end{align*}
For $ 0 \leq i \leq 2k - 2$, we define $\tau_{32 +i}$ by
\[
\tau_{32 + i} = \begin{cases}
r_5 + i/4 & i \equiv 0 \mod 4,\\
r_3 + \pi_{(i-1)/2} & i \equiv 1 , 3 \mod 4,\\
r_1 + (i-2)/4 & i \equiv 2 \mod 4,
\end{cases}
\]and we call this the \emph{middle section} of $\sigma$.
Note that so far we have only used the first $k - 1$ entries of $\pi$ and not $\pi_{k-1}$, but we will make use of the fact that the whole of $\pi$ was square-free when we check that $\sigma$ is square-free.

Let $\sigma$ be formed by appending one of the following vectors, and denote the length of the chosen suffix by $\ell$.
\begin{equation*}
    \begin{gathered}
    \left(r_4 + 1, r_4 , r_6 + 2, r_6 + 4, r_6 + 1, r_6, r_6 + 3 \right)\\ 
    \left(r_4 + 1, r_4, r_6 + 4, r_6 + 6, r_6 + 3, r_6 + 1, r_6 + 2, r_6 + 6, r_6\right)\\ 
    \left(r_4 + 1, r_4, r_6 + 3, r_6 + 4, r_6 + 2 , r_6, r_6 +1, r_6 +7, r_6 +6, r_6 +5, r_6 +8 \right)\\
    \left(r_4 +1, r_4, r_6 , r_6 + 1 , r_2 + 4, r_2 + 1, r_2 + 2, r_6 + 3, r_6 + 2, r_2 + 3, r_6 + 4, r_6 + 5, r_2\right)
    \end{gathered}
\end{equation*}
These have the patterns 
\begin{equation*}
    \begin{gathered}
    \left( 1, 0, 4, 6, 3, 2, 5\right)\\
    \left( 1,0,6,8,5,3,4,7,2\right)\\
    \left( 1,0,5,6,4,2,3,9,8,7,10\right)\\
    \left( 6,5,7,8,4,1,2,10,9,3,11,12,0\right)
    \end{gathered}
\end{equation*}
respectively. While the top two of these are right-crucial, the last two are not. From our construction, we know the last 7 entries of $\tau$ except for $\tau_{m - 6}$, $\tau_{m-4}$ and $\tau_{m-2}$. However, since $k \equiv 1 \bmod 4$ and $\pi$ begins up-down, we know that $\pi$ ends down-down-up and that $\tau_{m-6} > \tau_{m-4} > \tau_{m-2}$. Since these are the only entries in $R_3$ in the last 7 entries of $\tau$ and no entries in the suffixes are in $R_3$, this is enough to determine the pattern of the last $\ell + 7$ entries of $\sigma$. The possible patterns are 
\begin{equation*}
    \begin{gathered}
        \left( 0, 4, 7, 3, 1, 2, 8, 6, 5, 11, 13, 10, 9, 12\right)\\
        \left( 0, 4, 7, 3, 1, 2, 8, 6, 5, 13, 15, 12, 10, 11, 14, 9\right)\\
        \left(0, 4, 7, 3, 1, 2, 8, 6, 5, 12, 13, 11, 9, 10, 16, 15, 14, 17\right)\\
        \left( 0, 9, 12, 8, 1, 7, 13, 11, 10, 14, 15, 6, 3, 4, 17, 16, 5, 18, 19, 2\right),
    \end{gathered}
\end{equation*}
and these are all right-crucial (see Lemma \ref{lem:right_crucial}). Hence, any left or right extension of $\sigma$ contains a square and it only remains to argue that $\sigma$ is square-free.

First, note that by construction $\sigma$ follows the up-up-down-down pattern so does not contain any squares of length 4 and any square must have a length of at least 8. Let $S = (S_1; S_2)$ be a square. Since the middle section was constructed using the high-medium-low construction and is square-free, it must be the case that $S_1$ contains some of the first 32 entries, or that $S_2$ contains some of the last $\ell$.

Let us first consider the case where $S_1$ contains some of the first 32 entries, and further that $S_1$ is contained entirely within the first 32 entries. It is easy to check that the prefix does not contain the pattern $(3, 1, 0, 2)$ so it cannot possibly be the case that $S_2$ contains $(\tau_{32}, \dots \tau_{35})$. Hence, the square must be contained entirely within the first 35 entries, and the pattern of these is fixed for any choice of $\pi$ and does not contain a square. Suppose instead that $S_1$ is not contained entirely within the first 32 entries. If $S_1$ contains $(\tau_{31}, \tau_{32}, \tau_{33}, \tau_{34})$, then $S_2$ must contain the pattern $(0, 3, 2, 1)$. However, since $n > 94$, the corresponding entries in $S_2$ must come from the middle section of $\sigma$ and this never contains the pattern $(0, 3, 2, 1)$ as the first, medium entry is lower than the last, low entry. Hence, $S_1$ must end with either $\tau_{32}$ or $\tau_{33}$.

Suppose $S_1$ ends at $\tau_{33}$. It is easy to check that the square cannot be of length 8 as then $S_1$ would be order-isomorphic to $(0,1,3,2)$ while $S_2$ is order-isomorphic to $(0, 2, 3, 1)$. Hence, the square must be of length at least 16 and $S_1$ contains the pattern $(0, 4, 5, 3, 1, 2, 7, 6)$, but the third, high entry is lower than the final, medium entry. This pattern never occurs in the middle section of $\sigma$, while our choice of $n$ guarantees that the corresponding entries in $S_2$ come from the middle section. The case where $S_1$ ends at $\tau_{32}$ is similar. Explicitly, the square cannot be of length 8 as then $S_1$ is order-isomorphic to $(2,0,1,3)$ while $S_2$ is order-isomorphic to $(1,0,2,3)$. If the square is of length at least 16, then $S_1$ contains the pattern $(1,0,5, 6, 4, 2, 3, 7)$ which cannot occur in the middle section as the initial, medium value is lower than the sixth, low entry. Hence, $S_1$ does not intersect the first 32 entries.

Suppose $S_2$ intersects with the last $\ell$ entries. First, consider the case that $S_2$ is entirely contained in the last $\ell$ entries. Since $\ell \leq 13$, we know $S_2$ has length at most 12, and there are only two cases where the length of $S_2$ is $12$. These are 
\begin{equation*}
    \begin{gathered}
    (S_1; S_2) = (\sigma_{n - 24}, \dots, \sigma_{n-13}; \sigma_{n-12}, \dots, \sigma_{n-1})\\
    (S_1; S_2) = (\sigma_{n- 25}, \dots, \sigma_{n-14}; \sigma_{n-13}, \dots, \sigma_{n-2}).
    \end{gathered}
\end{equation*}
In both of these cases, the factor $(\sigma_{n-17}, \dots, \sigma_{n-14}) \oiso (2,0,1, 3)$ in $S_1$ corresponds to the factor $(\sigma_{n-5}, \dots, \sigma_{n-2}) \oiso (1, 0, 2, 3)$ in $S_2$. Hence, neither of these cases lead to squares and we can assume $S_2$ has length at most 8.
Since we know the pattern of the last $\ell + 7$ entries and it is square-free, we know $S_2$ must be the first 8 entries of the suffix (and $\ell \geq 8$). However, this means we have one of the following patterns.
\begin{equation*}
    \begin{gathered}
    \left( 1, 0, 5, 7, 4, 2, 3, 6\right)\\
    \left(1, 0, 5, 6, 4, 2, 3, 7\right)\\
    \left(4, 3, 5, 6, 2, 0, 1, 7\right)\\
    \end{gathered}
\end{equation*}
In each of these patterns, there is a medium entry which is lower than a low entry so these patterns never occur in the middle section of $\sigma$ and we get a contradiction.

Suppose instead that $S_2$ is not entirely contained in the last $\ell$ entries, but does intersect them. Then $S_2$ must contain at least one of the following.
\begin{equation*}
    \begin{gathered}
        \left( \tau_{m-1}, \tau_{m}, \tau_{m +1}, \tau_{m+2}\right)\\
        \left( \tau_{m-2}, \tau_{m-1}, \tau_{m}, \tau_{m+1}\right)\\
        \left( \tau_{m-3}, \tau_{m-2}, \tau_{m -1}, \tau_{m}\right)\\
    \end{gathered}
\end{equation*}

From our construction, we know that \[\left( \tau_{m-1}, \tau_{m}, \tau_{m +1}, \tau_{m+2}\right) = \left( r_6 - 1, r_4 + 1, r_4, r_6 + x\right)\] for some choice of $x$. In particular, the last, medium entry is higher than the first, high entry. For the second case we note that $\left( \tau_{m-2}, \tau_{m-1}, \tau_{m}, \tau_{m+1}\right) = (r_3 + y, r_6 - 1, r_4 + 1, r_4)$ for some suitable choice of $y$, and the first, medium entry is lower than the last, low entry. Finally, we have the case in which $S_2$ contains $(\tau_{m-3}, \tau_{m-2}, \tau_{m -1}, \tau_{m}) = (r_3 -1, r_3 + y, r_6 -1, r_4 + 1)$ for which we use the fact that $\pi_{k-1} = k-1$.  The permutation $(\tau_{32}, \dots, \tau_{m-1})$ follows the high-medium-low construction. If we continued using the high-medium-low construction, the entry $\tau_m$ would correspond to the entry $\pi_{k-1}$. Since we have chosen $\pi_{k-1} = k-1$, $\tau_m$ would be smaller than all the high entries, but greater than every other entry in $\{\tau_{32}, \dots, \tau_{m-1}\}$. In fact, this is exactly how we have chosen $\tau_m$, and $(\tau_{32}, \dots, \tau_{m})$ follows the high-medium-low construction as well. Hence, $\sigma$ is square-free.
\end{proof}

\section{Bicrucial permutations for small even $n$}
\label{sec:small-n}
It was shown by Gent, Kitaev, Konovalov, Linton and Nightingale \cite{gent2015crucial} that there are no bicrucial permutations of length $n$ when $n$ is even and $n < 32$, but that there is a bicrucial permutation when $n = 32$. Theorem \ref{thm:even} shows that there are bicrucial permutations of length $n$ for all even $n \geq 48$, so to prove Theorem \ref{thm:main} it only remains to consider the cases where $n$ is at least $34$ and at most $46$. 

To find examples of bicrucial permutations for $n= 34$ it is enough to consider permutations which begin with the initial pattern used in the proof of Theorem \ref{thm:even} and to try different ways of extending the permutation to the right, ensuring it is always square-free and checking if the resulting permutation is right-crucial. When $n = 40$ or $n = 42$, it suffices to use the construction from Theorem \ref{thm:even} with the fixed permutation used in the computer search. To find the example for $n = 46$, we looked for permutations that begin with the prefix pattern from Theorem \ref{thm:even} and also end with the relevant suffix pattern, but we did not enforce any other conditions such as how the prefix and suffix interleave. There is no bicrucial permutation of length $36$ which begins with the pattern from Theorem \ref{thm:even}. Instead, we tried extending other left-crucial permutations of length 32 until we found the example given below. The permutations can be seen in Figure \ref{fig:small-cases} in Appendix \ref{sec:small-n-pictures}.

\begin{align*}
    34&: (8, 2, 0, 21, 30, 20, 3, 4, 7, 5, 1, 31, 32, 22, 6, 9, 23, 12, 10, 25, 27, 24, 13, 14, 18, 15, 11,\\
    &\quad \quad 28, 29, 26, 17, 19, 33, 16).\\
    36&: (27, 33, 35, 22, 21, 26, 32, 31, 28, 30, 34, 19, 12, 20, 29, 25, 11, 18, 24, 9, 7, 10, 17, 16,\\
    &\quad\quad 13, 15, 23, 4, 2, 5, 14, 8, 1, 3, 6,  0).\\
    40&: (8, 2, 0, 20, 29, 19, 3, 4, 7, 5, 1, 30, 31, 21, 6, 9, 22, 12, 10, 24, 26, 23, 13, 14, 17, 15, 11,\\
    &\quad\quad 27, 28, 25, 16, 18, 34, 33, 32, 37, 39, 36, 35, 38).\\
    42&: (8, 2, 0, 20, 29, 19, 3, 4, 7, 5, 1, 30, 31, 21, 6, 9, 22, 12, 10, 24, 26, 23, 13, 14, 17, 15, 11,\\
    &\quad\quad 27, 28, 25, 16, 18, 34, 33, 32, 39, 41, 38, 36, 37, 40, 35).\\
    44&: (8, 2, 0, 20, 29, 19, 3, 4, 7, 5, 1, 30, 31, 21, 6, 9, 22, 12, 10, 24, 26, 23, 13, 14, 17, 15,11,\\
    &\quad\quad 27, 28, 25, 16, 18, 34, 33, 32, 38, 42, 37, 35, 36, 41, 40, 39, 43).\\
    46&: (8, 2, 0, 25, 34, 24, 3, 4, 7, 5, 1, 35, 36, 26, 6, 9, 27, 12, 10, 29, 31, 28, 13, 14, 17, 15, 11,\\
    &\quad\quad 32, 33, 30, 16, 23, 39, 38, 37, 40, 41, 22, 19, 20, 43, 42, 21, 44, 45, 18).
\end{align*}

\subsection{The non-existence of bicrucial permutations of length 38}
\label{sec:no-38}

To show that there are no bicrucial permutations of length 38 we used an exhaustive computer search which we describe in this section. The code can be found attached to the arXiv submission. 

Let $\sigma$ be a bicrucial permutation of length $n = 38$. Then the \emph{reverse} $\sigma'$ of $\sigma$ defined by $\sigma'_i = \sigma_{n - 1 -i}$ and the \emph{complement} $\tilde{\sigma}$ given by $\tilde{\sigma}_i = n - 1 - \sigma_i$ are both also bicrucial. Since $n$ is even and a square-free permutation must follow the up-up-down-down condition, either the start or end of $\sigma$ must follow the pattern up-up or down-down. By taking the complement if necessary we can assume the pattern is up-up, and by reversing the permutation if necessary, we can assume the permutation begins up-up.

The up-up-down-down condition also means that any square in an extension of $\sigma$ must either be of length 4 or a multiple of 8, and so every square in an extension of $\sigma$ has length at most 32. Hence, the prefix of length 31 must be left-crucial. It is possible to enumerate all 180,319,733 left-crucial permutations of length 31 that begin up-up, but simply extending these to find all such permutations of length 38 and then checking if the extended permutation is right-crucial is too computationally expensive. Instead, we relax our condition by discarding $k$ of the initial entries and try to extend the remaining entries to a right-crucial permutation of length $38 - k$. As long as $k \leq 7$, the extended permutations will have a length of at least 31 and therefore capture if the permutation can be right-crucial. Certainly, if there is a bicrucial permutation we will find its ending as a right-crucial permutation, but the existence of such a right-crucial permutation does not necessarily imply that there exists a bicrucial permutation of length 38 as we would still need to check for squares including the discarded entries. Fortunately, there turn out to be no such right-crucial permutations.

The main advantage of this method is that many of the left-crucial permutations end in the same pattern. Indeed, there are only 3,814,838 unique endings of length 24 from the 180,319,733 left-crucial permutations of length 31 that begin up-up. We are free to choose the length of the prefixes we generate and the value of $k$, although we only consider the case where $k = 7$. There are only 1,052,494,945 left-crucial permutations of length 32, a relatively modest increase of only 5.8 times. It turns out that nearly 80\% (844,852,238) of these are in fact bicrucial and there are only 1,350,082,610 left-crucial permutations of length 33 (which begin up-up), an increase of only 1.3 times. This pattern does not hold and there are 34,609,211,226 such permutations of length 34, an increase of 25.6 times. We therefore generate all left-crucial permutations of length 33 which begin up-up and discard the first 7 entries of each to get 24,763,327 unique endings, each of which we try to extend to a right-crucial permutation of length 31. Since we find no such extensions, there cannot be a bicrucial permutation of length $38$.

We generate permutations using a depth-first search. The children of a permutation $\sigma$ of length $n$ are the permutations of length $n + 1$ which begin with the pattern $\sigma$. This is easily done in practice. Given a permutation $\sigma$ of length $n$, the first child is formed by appending the entry $n$, and to move from the $i$th child to the next child we simply swap the entries $n - i + 1$ and $n - i$. Once we reach the last child, we can step backwards by removing the last entry and decrementing each entry. 

The advantage of this method of generating permutations is that we can easily prune unnecessary subtrees. For example, when we first visit a permutation we check if it contains a square, in which case we do not need to consider any children of the permutation and we can prune the subtree starting at this permutation.

When we first visit a permutation $\sigma$ of length $n \geq 3$ we also calculate a crude lower bound on the length of a left-crucial permutation $\tau$ which begins with the given pattern (if such a permutation exists), and we can reject any permutation where we know $\tau$ is of length at least 39. To calculate the bound we consider extending the permutation $\sigma$ to the left by prepending the entry $i$ for each choice of $i$ and checking if the extended permutation $\sigma^{(i)}$ contains a square. If it does, we move on to $i+1$. 
If there is not a square, we look for a ``partial square" in the permutation and compute a lower bound based on completing this square. Let $(n+1)/2 < k < n$ be the smallest multiple of 4 such that $(\sigma^{(i)}_0, \dots, \sigma^{(i)}_{k - 1})$ (obtained by extending $(\sigma_0,\dots,\sigma_{k-2})$ to the left by prepending $i$) is order-isomorphic to $(\sigma^{(i)}_{k}, \dots, \sigma^{(i)}_{n})$ (where we extend the definition of order-isomorphic to permutations of different lengths by replacing the larger permutation by its prefix of the appropriate length). Then the shortest square-free permutation beginning with $\sigma$ which contains a square when prepending $i$ is of length $2k -1$, and we can improve the lower bound to $2k-1$. If there is no such $k$, then the first part $S_1$ of a square $(S_1, S_2)$ must contain all of $\sigma^{(i)}$ and any square created by prepending $i$ must be of length at least $2m$ where $m = 4 \ceil{(n + 1)/4}$ is the smallest multiple of 4 at least $n+1$. This gives a lower bound of $2m - 1$. 

Suppose there are $i \neq j$ which both fall into this last case. The squares created by prepending $i$ and $j$ cannot both be of length $2m$, as there is a single value for $S_2$, the second half of the square, (for the given $m$) and it would need to be order-isomorphic to two different choices of $S_1$. Hence, we can improve the lower bound to $2(m + 4) -1$. More generally, we can count the number of times $t$ that $S_1$ must contain all of $\sigma'$, and we get the lower bound $2m + 8(t - 1) -1$.

Let us illustrate how we calculate the lower bound with an example. Let $\sigma = (0, 4, 5, 2, 1, 3)$. We first consider appending the entry $0$ to get the permutation $(0, 1, 5, 6, 3, 2, 4)$, which has a square $(0,1; 5, 6)$ and does not improve the lower bound. If we append the entry $1$, we get the permutation $(1, 0, 5, 6, 3, 2, 4)$. This does not contain a square of length 4, but $(1, 0, 5) \oiso (3, 2, 4)$, and we get the lower bound $7$. Prepending $2$, $3$ or $4$ works similarly, and also give the lower bound $7$. Consider appending the entry $5$ to get $\sigma^{(5)} = (5, 0, 4, 6, 2, 1, 3)$. The only possible choice for $k$ is 4, but $(5, 0, 4) \not \oiso (2, 1, 3)$. Hence, if $\tau$ is square-free but contains a square when prepending 4, then the length of $\tau$ is at least 15. There is also no viable $k$ when prepending $6$, and so we get the improved lower bound 23.

\subsection{Enumerating square-free permutations}
\label{sec:enumeration}
The code used to exhaustively generate all left-crucial permutations in the previous section can also be used to enumerate square-free, left-crucial and bicrucial permutations with only small modifications, and the results are shown in Table \ref{tab:square-free}. If $\sigma$ is square-free (or bicrucial), then the reverse, the complement and the reverse complement of $\sigma$ are also square-free (resp. bicrucial). When $n$ is even, this means we can count the number of square-free permutations by counting the square-free permutations which begin up-up and multiplying by 4. When $n$ is odd, there is not such a nice classification and it is harder to handle the reverse of the permutation. To avoid this we check only that each permutation in the depth-first search is square-free until the permutation is of length at least $n/2$, and only then do we consider exploiting the symmetry in the problem. Note that any extra work done generating square-free permutations of length $n/2$ is negligible when compared with the work to generate the permutations of length $n$.  When appending an entry to grow the permutation to length $(n+1)/2 + k$, we consider the factor 
\[
\left( \sigma_{(n+1)/2 - k}, \dots, \sigma_{(n+1)/2 + k} \right) \oiso \pi
\]
and check if $\pi$, the reverse of $\pi$, the complement of $\pi$ or the reverse complement of $\pi$ is lexicographically least. If $\pi$ is not jointly lexicographically least, then we prune the permutation from our search. If $\pi$ is the unique minimum, then we do not perform this check again when appending new entries. Multiplying the number of permutations found this way by 4 over counts permutations $\sigma$ where the reverse and the complement are equal, and we must subtract off twice the number of such permutations (once for the reverse and once for the reverse complement). 

Enumerating left-crucial permutations is slightly simpler: only the complement of a left-crucial permutation need be left-crucial, so we can handle the symmetry by only counting the left-crucial permutations which begin by going up and multiplying this result by 2. 

\begin{table}
    \centering
    \begin{tabular}{cn{15}{0}n{14}{0}n{14}{0}}\toprule
	$n$ & {\thead{Square-free}} & {\thead{Left-crucial}} & {\thead{Bicrucial}} \\
	\midrule
	1 & 1 & 0 &0\\
	2 & 2 & 0 &0\\
	3 & 6 & 0 &0\\
	4 & 12& 0 &0\\
	5 & 34& 0 &0\\
	6 & 104& 0 &0\\
	7 & 406& 60 &0\\
	8 & 1112& 140 &0\\
	9 & 3980& 518 & 54\\
	10 & 15216& 1444 & 0\\
	11 & 68034& 8556 & 0 \\
	12 & 312048& 31992 &0 \\
	13 & 1625968& 220456 & 69856 \\
	14 & 8771376& 984208 &0 \\
	15 &53270068& 7453080 & 2930016\\
	16 &319218912 & 39692800 &0\\
	17 &2135312542 & 289981136 & 40654860\\
	18 & 14420106264& 1467791790 &0\\
	19 & 109051882344 & 14316379108 & 162190472\\
	20 & 815868128288 & 86001855074 &0\\
	21 & 6772099860398 & 949804475890 & 312348610684\\
	22 & 56501841264216 &  6494842788046&0\\
	23 & 519359404861294 & 73636377696714 & 29202730580288\\
	\bottomrule
\end{tabular}
    \caption{The number of square-free, left-crucial and bicrucial permutations of length $n$. These are sequences A221989, A221990 and A238935 respectively in the OEIS \cite{oeis}.}
    \label{tab:square-free}
\end{table}

\section{Open problems}

Recently, the notion of bicrucial square-free words has been extended to the notion of \emph{extremal words}, words which are square-free but inserting any letter in any position introduces a square \cite{grytczuk2020extremal}, and this naturally extends to \emph{extremal permutations}. Since square-free permutations must follow the up-up-down-down pattern, any insertion which breaks this pattern must introduce a square and so any insertion except in positions $0$, $1$, $n-1$ or $n$ introduces a square. Extremal permutations were studied by Gent, Kitaev, Konovalov, Linton and Nightingale in \cite{gent2015crucial} where they showed that there exist small extremal permutations. In particular, they showed that there exist extremal permutations of lengths 17 and 21 but not for any other $n \leq 22$, and they conjectured that arbitrarily long extremal permutations exist. Since one only needs to consider inserting an entry close to the beginning or end as in bicrucial permutations, it seems likely that an approach similar to that used in Theorem \ref{thm:8k+3} and Theorem \ref{thm:even} may give arbitrarily long extremal permutations.

\begin{conjecture}[Conjecture 10 in \cite{gent2015crucial}]
There exist arbitrarily long extremal square-free permutations.
\end{conjecture}

It was shown by Grytczuk, Kordulewski and Niewiadomski \cite{grytczuk2020extremal} that there are arbitrarily long extremal square-free words over an alphabet of size 3, and the $n$ for which there exist extremal words of length $n$ was completely classified by Mol and Rampersad \cite{mol2020lengths}. In particular, there are extremal ternary words for all $n \geq 87$. There are no known extremal words over an alphabet of size $4$ (or greater), and it is conjectured that none exist.

\begin{conjecture}[Conjecture 12 in \cite{grytczuk2020extremal}]
There are no extremal square-free words over an alphabet of size 4.
\end{conjecture}

It was recently shown that there are no extremal square-free words over an alphabet of size 17 (or greater) \cite{hong2021no}, but the question remains open for smaller alphabets.

Another interesting problem suggested by Kitaev \cite{Kitaev} is to replace squares with higher powers. A \emph{$k$th power} in a permutation $\sigma\in S_n$ is a factor
\[
(S_1;S_2;\dots;S_k)=(\sigma_{s},\dots,\sigma_{s+\ell-1};\dots ;\sigma_{s+(k-1)\ell},\dots,\sigma_{s+k\ell-1})
\]
with $\ell\geq 2$ and $S_i \oiso S_j$ for all $i,j\in [k]$. The definition of bicrucial with respect to squares easily extends to bicrucial with respect to containing a $k$th power. That is, we say the permutation $\sigma$ is \emph{bicrucial with respect to containing a $k$th power} if $\sigma$ does not contain a $k$th power, but any left or right extension of $\sigma$ does contain a $k$th power.

\begin{problem}
    Are there arbitrarily long permutations which are bicrucial with respect to containing a $k$th power?
\end{problem}

\paragraph{Acknowledgements.} We would like to thank the two anonymous referees for their helpful comments which have improved this paper.

\bibliographystyle{abbrv}
\bibliography{bicrucial}

\appendix

\section{Small crucial permutations}
\label{sec:small-ends}
For the convenience of the reader, we add proofs here for some of the more routine checks from the proof of Theorem \ref{thm:even}.
\begin{lemma}
\label{lem:left_crucial}
	The following permutation is left-crucial.
	\begin{align*}
		\sigma &= (8, 2, 0, 20, 29, 19, 3, 4, 7, 5, 1, 30, 31, 21, 6, 9, 22, 12, 10, 24, 26, 23, 13, 14, 17, 15, 11,\\ &\quad \quad 27, 28, 25, 16, 18).
	\end{align*}
\end{lemma}
\begin{proof}
	We start with the easiest case, prepending $x \geq 9$. In this case, the extended permutation begins with the pattern $(3, 2, 1, 0)$ and so contains a square of length 4. If we instead prepend $3 \leq x \leq 8$, the permutation begins with \[\left(x, 9, 2, 0, 21, 30, 20, y\right),\]
	where $y=4$ if $x=3$ and $y=3$ if $4\leq x\leq 8$. This contains a square since $(x, 9, 2, 0) \sim (2, 3, 1, 0) \sim (20, 29, 19, y)$. Prepending $1 \leq x \leq 2$ gives a square of length 16. Indeed, both the first 8 and the second 8 entries both have the pattern $(1, 4, 2, 0, 6, 7, 5, 3)$. Finally, prepending $0$ gives a square of length 32.
\end{proof}

\begin{lemma}	
\label{lem:right_crucial}
	The following permutations are right-crucial.	
		\begin{align*}
			\sigma^{(1)} &= \left( 0, 4, 7, 3, 1, 2, 8, 6, 5, 11, 13, 10, 9, 12\right)\\
			\sigma^{(2)} &= \left( 0, 4, 7, 3, 1, 2, 8, 6, 5, 13, 15, 12, 10, 11, 14, 9\right)\\
			\sigma^{(3)} &= \left(0, 4, 7, 3, 1, 2, 8, 6, 5, 12, 13, 11, 9, 10, 16, 15, 14, 17\right)\\
			\sigma^{(4)} &= \left( 0, 9, 12, 8, 1, 7, 13, 11, 10, 14, 15, 6, 3, 4, 17, 16, 5, 18, 19, 2\right).
		\end{align*}
\end{lemma}
\begin{proof}[Sketch proof]
	We only specify the lengths of squares when appending different values and leave the straightforward task of checking that there are indeed squares of the given lengths to the reader.
	For example, the table below says that adding $x\leq 12$ to the end of $\sigma^{(1)}$ (and increasing all values that were previously at least $x$ by $1$) creates a square of length $4$, whereas adding $x\geq 13$ gives a square of length 8.\\
	
			\begin{center}
			\begin{tabular}{@{}cccc@{}}
				\toprule
				& \multicolumn{3}{c}{Square length}\\
				\cmidrule(r){2-4}
				Permutation & 4 & 8 & 16\\
				\midrule
				$\sigma^{(1)}$ & $\leq 12$  & $\geq 13$ &   \\
				$\sigma^{(2)}$ & $\geq 9$  & $\leq 8$ &   \\
				$\sigma^{(3)}$ & $\leq 17$  & & 18   \\
				$\sigma^{(4)}$ & $\leq 2$  & &$\geq 3$\\
				\bottomrule
			\end{tabular}
			\end{center}
			\par \vspace{-1.6\baselineskip} 
			\qedhere
\end{proof}
\section{Constructions from Theorem \ref{thm:even}}
\input{even}

\section{Bicrucial permutations for small even $n$}
\label{sec:small-n-pictures}
\input{small-cases}

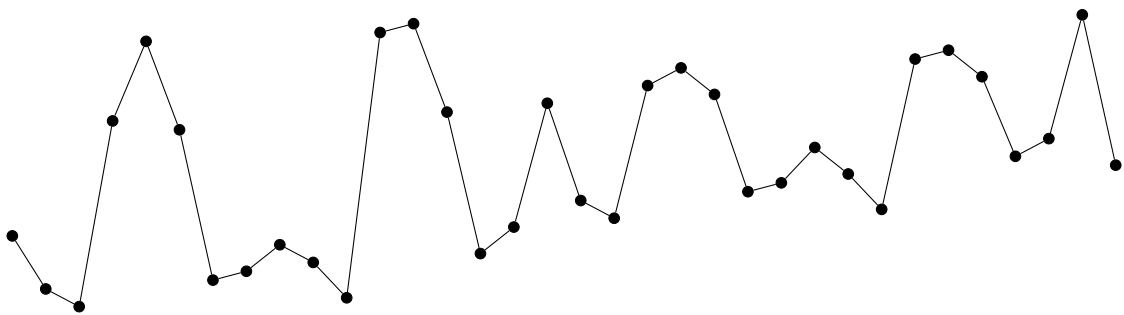
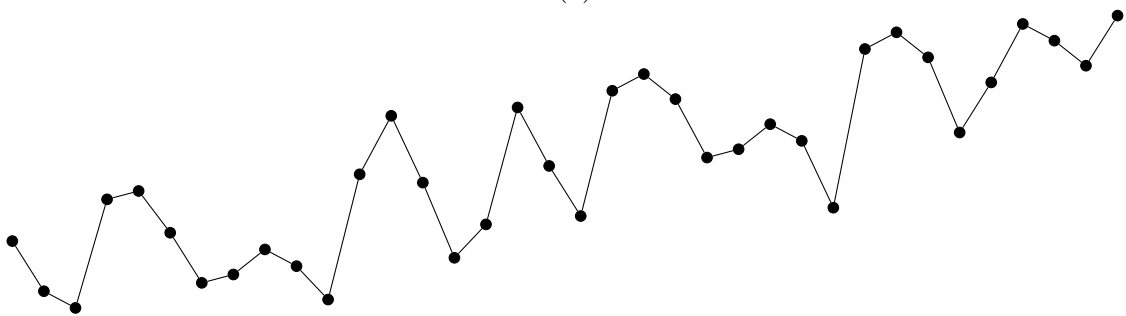
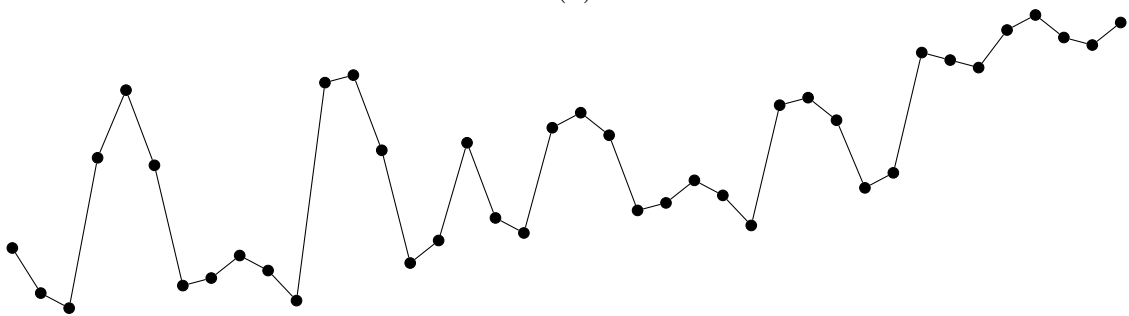
\end{document}

%% file: 72.tex
\begin{figure}
\centering
\begin{tikzpicture}[xscale=\textwidth/71cm, yscale=8cm/124 cm]
\draw node[style=vertex](0) at (0,8) {};
\draw node[style=vertex](1) at (1,2) {};
\draw node[style=vertex](2) at (2,0) {};
\draw node[style=vertex](3) at (3,20) {};
\draw node[style=vertex](4) at (4,29) {};
\draw node[style=vertex](5) at (5,19) {};
\draw node[style=vertex](6) at (6,3) {};
\draw node[style=vertex](7) at (7,4) {};
\draw node[style=vertex](8) at (8,7) {};
\draw node[style=vertex](9) at (9,5) {};
\draw node[style=vertex](10) at (10,1) {};
\draw node[style=vertex](11) at (11,30) {};
\draw node[style=vertex](12) at (12,31) {};
\draw node[style=vertex](13) at (13,21) {};
\draw node[style=vertex](14) at (14,6) {};
\draw node[style=vertex](15) at (15,9) {};
\draw node[style=vertex](16) at (16,22) {};
\draw node[style=vertex](17) at (17,12) {};
\draw node[style=vertex](18) at (18,10) {};
\draw node[style=vertex](19) at (19,24) {};
\draw node[style=vertex](20) at (20,26) {};
\draw node[style=vertex](21) at (21,23) {};
\draw node[style=vertex](22) at (22,13) {};
\draw node[style=vertex](23) at (23,14) {};
\draw node[style=vertex](24) at (24,17) {};
\draw node[style=vertex](25) at (25,15) {};
\draw node[style=vertex](26) at (26,11) {};
\draw node[style=vertex](27) at (27,27) {};
\draw node[style=vertex](28) at (28,28) {};
\draw node[style=vertex](29) at (29,25) {};
\draw node[style=vertex](30) at (30,16) {};
\draw node[style=vertex](31) at (31,18) {};
\draw node[style=vertex](32) at (32,102) {};
\draw node[style=vertex](33) at (33,74) {};
\draw node[style=vertex](34) at (34,41) {};
\draw node[style=vertex](35) at (35,79) {};
\draw node[style=vertex](36) at (36,103) {};
\draw node[style=vertex](37) at (37,77) {};
\draw node[style=vertex](38) at (38,42) {};
\draw node[style=vertex](39) at (39,70) {};
\draw node[style=vertex](40) at (40,104) {};
\draw node[style=vertex](41) at (41,76) {};
\draw node[style=vertex](42) at (42,43) {};
\draw node[style=vertex](43) at (43,80) {};
\draw node[style=vertex](44) at (44,105) {};
\draw node[style=vertex](45) at (45,73) {};
\draw node[style=vertex](46) at (46,44) {};
\draw node[style=vertex](47) at (47,71) {};
\draw node[style=vertex](48) at (48,106) {};
\draw node[style=vertex](49) at (49,75) {};
\draw node[style=vertex](50) at (50,45) {};
\draw node[style=vertex](51) at (51,81) {};
\draw node[style=vertex](52) at (52,107) {};
\draw node[style=vertex](53) at (53,78) {};
\draw node[style=vertex](54) at (54,46) {};
\draw node[style=vertex](55) at (55,72) {};
\draw node[style=vertex](56) at (56,108) {};
\draw node[style=vertex](57) at (57,92) {};
\draw node[style=vertex](58) at (58,91) {};
\draw node[style=vertex](59) at (59,118) {};
\draw node[style=vertex](60) at (60,119) {};
\draw node[style=vertex](61) at (61,60) {};
\draw node[style=vertex](62) at (62,57) {};
\draw node[style=vertex](63) at (63,58) {};
\draw node[style=vertex](64) at (64,121) {};
\draw node[style=vertex](65) at (65,120) {};
\draw node[style=vertex](66) at (66,59) {};
\draw node[style=vertex](67) at (67,122) {};
\draw node[style=vertex](68) at (68,123) {};
\draw node[style=vertex](69) at (69,56) {};
\draw (0) -- (1) -- (2) -- (3) -- (4) -- (5) -- (6) -- (7) -- (8) -- (9) -- (10) -- (11) -- (12) -- (13) -- (14) -- (15) -- (16) -- (17) -- (18) -- (19) -- (20) -- (21) -- (22) -- (23) -- (24) -- (25) -- (26) -- (27) -- (28) -- (29) -- (30) -- (31) -- (32) -- (33) -- (34) -- (35) -- (36) -- (37) -- (38) -- (39) -- (40) -- (41) -- (42) -- (43) -- (44) -- (45) -- (46) -- (47) -- (48) -- (49) -- (50) -- (51) -- (52) -- (53) -- (54) -- (55) -- (56) -- (57) -- (58) -- (59) -- (60) -- (61) -- (62) -- (63) -- (64) -- (65) -- (66) -- (67) -- (68) -- (69) ;

\draw[dashed] (-1, 36) -- (70, 36);
\draw[dashed] (-1, 51) -- (70, 51);
\draw[dashed] (-1, 65) -- (70, 65);
\draw[dashed] (-1, 86) -- (70, 86);
\draw[dashed] (-1, 97) -- (70, 97);
\draw[dashed] (-1, 113) -- (70, 113);
\end{tikzpicture}
\caption{A bicrucial permutation of length 72 constructed as in the proof of Theorem \ref{thm:even}. The dashed lines divide $\{0, \dots, 71\}$ into the 7 different regions used in the construction.}
\label{fig:72}
\end{figure}
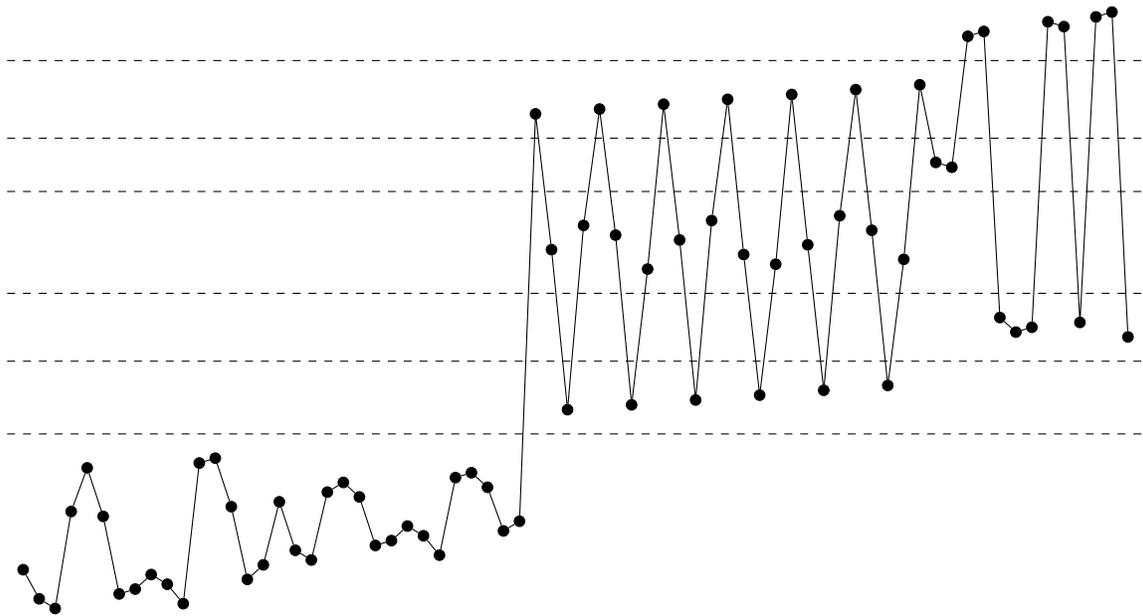

%% file: even.tex
\begin{figure}[h]
    \centering
    \begin{subfigure}{\textwidth}
    \begin{tikzpicture}[xscale=\textwidth/65cm, yscale=4cm/118cm]
\draw node[style=vertex](0) at (0,8) {};
\draw node[style=vertex](1) at (1,2) {};
\draw node[style=vertex](2) at (2,0) {};
\draw node[style=vertex](3) at (3,20) {};
\draw node[style=vertex](4) at (4,29) {};
\draw node[style=vertex](5) at (5,19) {};
\draw node[style=vertex](6) at (6,3) {};
\draw node[style=vertex](7) at (7,4) {};
\draw node[style=vertex](8) at (8,7) {};
\draw node[style=vertex](9) at (9,5) {};
\draw node[style=vertex](10) at (10,1) {};
\draw node[style=vertex](11) at (11,30) {};
\draw node[style=vertex](12) at (12,31) {};
\draw node[style=vertex](13) at (13,21) {};
\draw node[style=vertex](14) at (14,6) {};
\draw node[style=vertex](15) at (15,9) {};
\draw node[style=vertex](16) at (16,22) {};
\draw node[style=vertex](17) at (17,12) {};
\draw node[style=vertex](18) at (18,10) {};
\draw node[style=vertex](19) at (19,24) {};
\draw node[style=vertex](20) at (20,26) {};
\draw node[style=vertex](21) at (21,23) {};
\draw node[style=vertex](22) at (22,13) {};
\draw node[style=vertex](23) at (23,14) {};
\draw node[style=vertex](24) at (24,17) {};
\draw node[style=vertex](25) at (25,15) {};
\draw node[style=vertex](26) at (26,11) {};
\draw node[style=vertex](27) at (27,27) {};
\draw node[style=vertex](28) at (28,28) {};
\draw node[style=vertex](29) at (29,25) {};
\draw node[style=vertex](30) at (30,16) {};
\draw node[style=vertex](31) at (31,18) {};
\draw node[style=vertex](32) at (32,97) {};
\draw node[style=vertex](33) at (33,69) {};
\draw node[style=vertex](34) at (34,41) {};
\draw node[style=vertex](35) at (35,74) {};
\draw node[style=vertex](36) at (36,98) {};
\draw node[style=vertex](37) at (37,72) {};
\draw node[style=vertex](38) at (38,42) {};
\draw node[style=vertex](39) at (39,65) {};
\draw node[style=vertex](40) at (40,99) {};
\draw node[style=vertex](41) at (41,71) {};
\draw node[style=vertex](42) at (42,43) {};
\draw node[style=vertex](43) at (43,75) {};
\draw node[style=vertex](44) at (44,100) {};
\draw node[style=vertex](45) at (45,68) {};
\draw node[style=vertex](46) at (46,44) {};
\draw node[style=vertex](47) at (47,66) {};
\draw node[style=vertex](48) at (48,101) {};
\draw node[style=vertex](49) at (49,70) {};
\draw node[style=vertex](50) at (50,45) {};
\draw node[style=vertex](51) at (51,76) {};
\draw node[style=vertex](52) at (52,102) {};
\draw node[style=vertex](53) at (53,73) {};
\draw node[style=vertex](54) at (54,46) {};
\draw node[style=vertex](55) at (55,67) {};
\draw node[style=vertex](56) at (56,103) {};
\draw node[style=vertex](57) at (57,87) {};
\draw node[style=vertex](58) at (58,86) {};
\draw node[style=vertex](59) at (59,115) {};
\draw node[style=vertex](60) at (60,117) {};
\draw node[style=vertex](61) at (61,114) {};
\draw node[style=vertex](62) at (62,113) {};
\draw node[style=vertex](63) at (63,116) {};
\draw (0) -- (1) -- (2) -- (3) -- (4) -- (5) -- (6) -- (7) -- (8) -- (9) -- (10) -- (11) -- (12) -- (13) -- (14) -- (15) -- (16) -- (17) -- (18) -- (19) -- (20) -- (21) -- (22) -- (23) -- (24) -- (25) -- (26) -- (27) -- (28) -- (29) -- (30) -- (31) -- (32) -- (33) -- (34) -- (35) -- (36) -- (37) -- (38) -- (39) -- (40) -- (41) -- (42) -- (43) -- (44) -- (45) -- (46) -- (47) -- (48) -- (49) -- (50) -- (51) -- (52) -- (53) -- (54) -- (55) -- (56) -- (57) -- (58) -- (59) -- (60) -- (61) -- (62) -- (63) ;

\draw[dashed] (-1, 36) -- (64, 36);
\draw[dashed] (-1, 51) -- (64, 51);
\draw[dashed] (-1, 60) -- (64, 60);
\draw[dashed] (-1, 81) -- (64, 81);
\draw[dashed] (-1, 92) -- (64, 92);
\draw[dashed] (-1, 108) -- (64, 108);
\end{tikzpicture}
\caption{}
\end{subfigure}

\begin{subfigure}{\textwidth}
\begin{tikzpicture}[xscale=\textwidth/67cm, yscale=4cm/120cm]
\draw node[style=vertex](0) at (0,8) {};
\draw node[style=vertex](1) at (1,2) {};
\draw node[style=vertex](2) at (2,0) {};
\draw node[style=vertex](3) at (3,20) {};
\draw node[style=vertex](4) at (4,29) {};
\draw node[style=vertex](5) at (5,19) {};
\draw node[style=vertex](6) at (6,3) {};
\draw node[style=vertex](7) at (7,4) {};
\draw node[style=vertex](8) at (8,7) {};
\draw node[style=vertex](9) at (9,5) {};
\draw node[style=vertex](10) at (10,1) {};
\draw node[style=vertex](11) at (11,30) {};
\draw node[style=vertex](12) at (12,31) {};
\draw node[style=vertex](13) at (13,21) {};
\draw node[style=vertex](14) at (14,6) {};
\draw node[style=vertex](15) at (15,9) {};
\draw node[style=vertex](16) at (16,22) {};
\draw node[style=vertex](17) at (17,12) {};
\draw node[style=vertex](18) at (18,10) {};
\draw node[style=vertex](19) at (19,24) {};
\draw node[style=vertex](20) at (20,26) {};
\draw node[style=vertex](21) at (21,23) {};
\draw node[style=vertex](22) at (22,13) {};
\draw node[style=vertex](23) at (23,14) {};
\draw node[style=vertex](24) at (24,17) {};
\draw node[style=vertex](25) at (25,15) {};
\draw node[style=vertex](26) at (26,11) {};
\draw node[style=vertex](27) at (27,27) {};
\draw node[style=vertex](28) at (28,28) {};
\draw node[style=vertex](29) at (29,25) {};
\draw node[style=vertex](30) at (30,16) {};
\draw node[style=vertex](31) at (31,18) {};
\draw node[style=vertex](32) at (32,97) {};
\draw node[style=vertex](33) at (33,69) {};
\draw node[style=vertex](34) at (34,41) {};
\draw node[style=vertex](35) at (35,74) {};
\draw node[style=vertex](36) at (36,98) {};
\draw node[style=vertex](37) at (37,72) {};
\draw node[style=vertex](38) at (38,42) {};
\draw node[style=vertex](39) at (39,65) {};
\draw node[style=vertex](40) at (40,99) {};
\draw node[style=vertex](41) at (41,71) {};
\draw node[style=vertex](42) at (42,43) {};
\draw node[style=vertex](43) at (43,75) {};
\draw node[style=vertex](44) at (44,100) {};
\draw node[style=vertex](45) at (45,68) {};
\draw node[style=vertex](46) at (46,44) {};
\draw node[style=vertex](47) at (47,66) {};
\draw node[style=vertex](48) at (48,101) {};
\draw node[style=vertex](49) at (49,70) {};
\draw node[style=vertex](50) at (50,45) {};
\draw node[style=vertex](51) at (51,76) {};
\draw node[style=vertex](52) at (52,102) {};
\draw node[style=vertex](53) at (53,73) {};
\draw node[style=vertex](54) at (54,46) {};
\draw node[style=vertex](55) at (55,67) {};
\draw node[style=vertex](56) at (56,103) {};
\draw node[style=vertex](57) at (57,87) {};
\draw node[style=vertex](58) at (58,86) {};
\draw node[style=vertex](59) at (59,117) {};
\draw node[style=vertex](60) at (60,119) {};
\draw node[style=vertex](61) at (61,116) {};
\draw node[style=vertex](62) at (62,114) {};
\draw node[style=vertex](63) at (63,115) {};
\draw node[style=vertex](64) at (64,118) {};
\draw node[style=vertex](65) at (65,113) {};
\draw (0) -- (1) -- (2) -- (3) -- (4) -- (5) -- (6) -- (7) -- (8) -- (9) -- (10) -- (11) -- (12) -- (13) -- (14) -- (15) -- (16) -- (17) -- (18) -- (19) -- (20) -- (21) -- (22) -- (23) -- (24) -- (25) -- (26) -- (27) -- (28) -- (29) -- (30) -- (31) -- (32) -- (33) -- (34) -- (35) -- (36) -- (37) -- (38) -- (39) -- (40) -- (41) -- (42) -- (43) -- (44) -- (45) -- (46) -- (47) -- (48) -- (49) -- (50) -- (51) -- (52) -- (53) -- (54) -- (55) -- (56) -- (57) -- (58) -- (59) -- (60) -- (61) -- (62) -- (63) -- (64) -- (65) ;
\draw[dashed] (-1, 36) -- (66, 36);
\draw[dashed] (-1, 51) -- (66, 51);
\draw[dashed] (-1, 60) -- (66, 60);
\draw[dashed] (-1, 81) -- (66, 81);
\draw[dashed] (-1, 92) -- (66, 92);
\draw[dashed] (-1, 108) -- (66, 108);
\end{tikzpicture}
\caption{}
\end{subfigure}

\begin{subfigure}{\textwidth}
\begin{tikzpicture}[xscale=\textwidth/69cm, yscale=4cm/122cm]
\draw node[style=vertex](0) at (0,8) {};
\draw node[style=vertex](1) at (1,2) {};
\draw node[style=vertex](2) at (2,0) {};
\draw node[style=vertex](3) at (3,20) {};
\draw node[style=vertex](4) at (4,29) {};
\draw node[style=vertex](5) at (5,19) {};
\draw node[style=vertex](6) at (6,3) {};
\draw node[style=vertex](7) at (7,4) {};
\draw node[style=vertex](8) at (8,7) {};
\draw node[style=vertex](9) at (9,5) {};
\draw node[style=vertex](10) at (10,1) {};
\draw node[style=vertex](11) at (11,30) {};
\draw node[style=vertex](12) at (12,31) {};
\draw node[style=vertex](13) at (13,21) {};
\draw node[style=vertex](14) at (14,6) {};
\draw node[style=vertex](15) at (15,9) {};
\draw node[style=vertex](16) at (16,22) {};
\draw node[style=vertex](17) at (17,12) {};
\draw node[style=vertex](18) at (18,10) {};
\draw node[style=vertex](19) at (19,24) {};
\draw node[style=vertex](20) at (20,26) {};
\draw node[style=vertex](21) at (21,23) {};
\draw node[style=vertex](22) at (22,13) {};
\draw node[style=vertex](23) at (23,14) {};
\draw node[style=vertex](24) at (24,17) {};
\draw node[style=vertex](25) at (25,15) {};
\draw node[style=vertex](26) at (26,11) {};
\draw node[style=vertex](27) at (27,27) {};
\draw node[style=vertex](28) at (28,28) {};
\draw node[style=vertex](29) at (29,25) {};
\draw node[style=vertex](30) at (30,16) {};
\draw node[style=vertex](31) at (31,18) {};
\draw node[style=vertex](32) at (32,97) {};
\draw node[style=vertex](33) at (33,69) {};
\draw node[style=vertex](34) at (34,41) {};
\draw node[style=vertex](35) at (35,74) {};
\draw node[style=vertex](36) at (36,98) {};
\draw node[style=vertex](37) at (37,72) {};
\draw node[style=vertex](38) at (38,42) {};
\draw node[style=vertex](39) at (39,65) {};
\draw node[style=vertex](40) at (40,99) {};
\draw node[style=vertex](41) at (41,71) {};
\draw node[style=vertex](42) at (42,43) {};
\draw node[style=vertex](43) at (43,75) {};
\draw node[style=vertex](44) at (44,100) {};
\draw node[style=vertex](45) at (45,68) {};
\draw node[style=vertex](46) at (46,44) {};
\draw node[style=vertex](47) at (47,66) {};
\draw node[style=vertex](48) at (48,101) {};
\draw node[style=vertex](49) at (49,70) {};
\draw node[style=vertex](50) at (50,45) {};
\draw node[style=vertex](51) at (51,76) {};
\draw node[style=vertex](52) at (52,102) {};
\draw node[style=vertex](53) at (53,73) {};
\draw node[style=vertex](54) at (54,46) {};
\draw node[style=vertex](55) at (55,67) {};
\draw node[style=vertex](56) at (56,103) {};
\draw node[style=vertex](57) at (57,87) {};
\draw node[style=vertex](58) at (58,86) {};
\draw node[style=vertex](59) at (59,116) {};
\draw node[style=vertex](60) at (60,117) {};
\draw node[style=vertex](61) at (61,115) {};
\draw node[style=vertex](62) at (62,113) {};
\draw node[style=vertex](63) at (63,114) {};
\draw node[style=vertex](64) at (64,120) {};
\draw node[style=vertex](65) at (65,119) {};
\draw node[style=vertex](66) at (66,118) {};
\draw node[style=vertex](67) at (67,121) {};
\draw (0) -- (1) -- (2) -- (3) -- (4) -- (5) -- (6) -- (7) -- (8) -- (9) -- (10) -- (11) -- (12) -- (13) -- (14) -- (15) -- (16) -- (17) -- (18) -- (19) -- (20) -- (21) -- (22) -- (23) -- (24) -- (25) -- (26) -- (27) -- (28) -- (29) -- (30) -- (31) -- (32) -- (33) -- (34) -- (35) -- (36) -- (37) -- (38) -- (39) -- (40) -- (41) -- (42) -- (43) -- (44) -- (45) -- (46) -- (47) -- (48) -- (49) -- (50) -- (51) -- (52) -- (53) -- (54) -- (55) -- (56) -- (57) -- (58) -- (59) -- (60) -- (61) -- (62) -- (63) -- (64) -- (65) -- (66) -- (67) ;
\draw[dashed] (-1, 36) -- (68, 36);
\draw[dashed] (-1, 51) -- (68, 51);
\draw[dashed] (-1, 60) -- (68, 60);
\draw[dashed] (-1, 81) -- (68, 81);
\draw[dashed] (-1, 92) -- (68, 92);
\draw[dashed] (-1, 108) -- (68, 108);
\end{tikzpicture}
\caption{}
\end{subfigure}

\caption{Bicrucial permutations of length 64, 66 and 68 constructed as in the proof of Theorem \ref{thm:even}. The dashed lines divide $\{0, \dots, n-1\}$ into the 7 different regions used in the construction.}
\vspace{-1.5cm}
\end{figure}
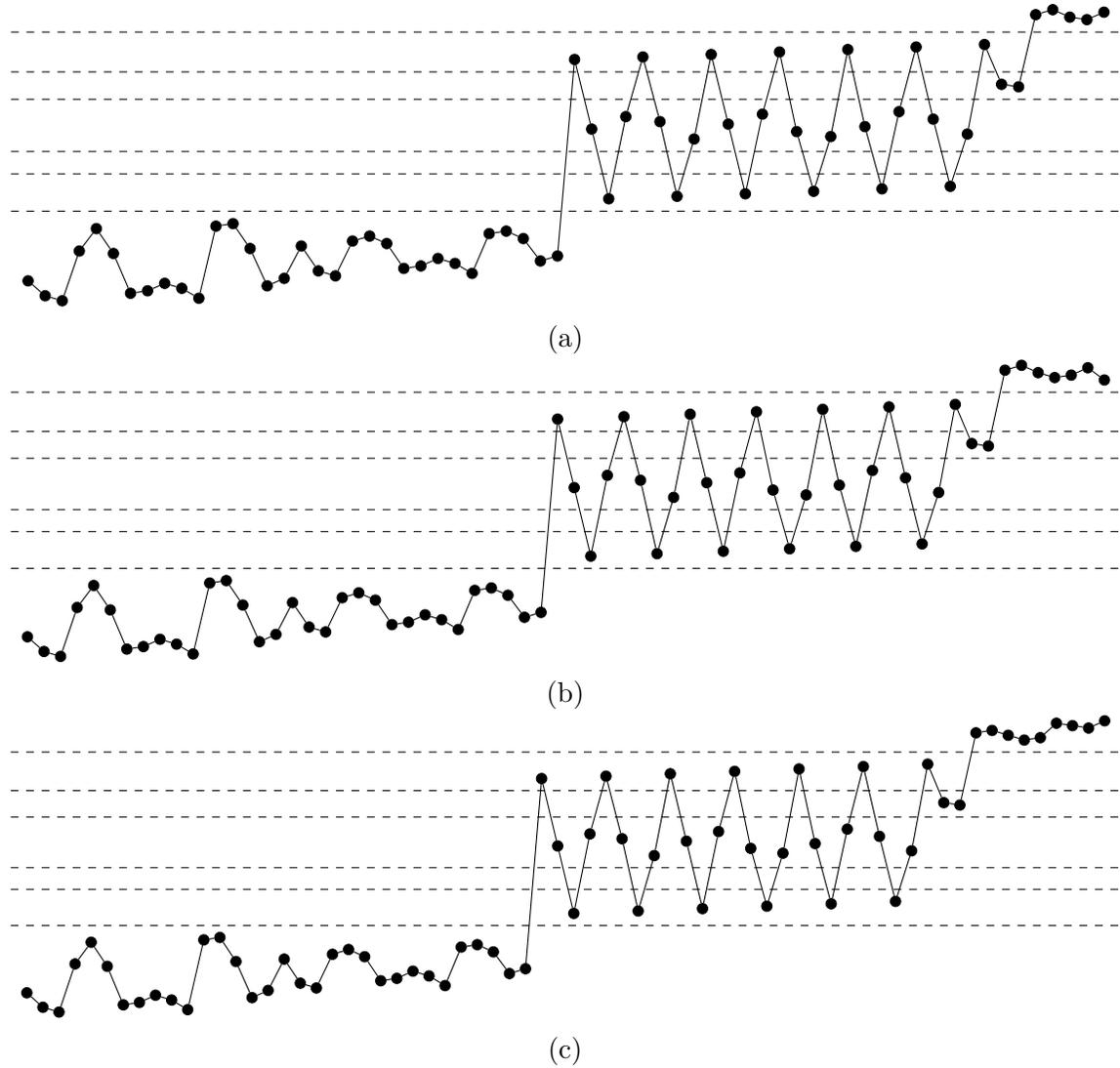
\newpage

%% file: small-cases.tex
\begin{figure}[h!]
\begin{subfigure}{\textwidth}
\begin{tikzpicture}[xscale=\textwidth/34cm, yscale=4cm/34cm]
\draw node[style=vertex](0) at (0,8) {};
\draw node[style=vertex](1) at (1,2) {};
\draw node[style=vertex](2) at (2,0) {};
\draw node[style=vertex](3) at (3,21) {};
\draw node[style=vertex](4) at (4,30) {};
\draw node[style=vertex](5) at (5,20) {};
\draw node[style=vertex](6) at (6,3) {};
\draw node[style=vertex](7) at (7,4) {};
\draw node[style=vertex](8) at (8,7) {};
\draw node[style=vertex](9) at (9,5) {};
\draw node[style=vertex](10) at (10,1) {};
\draw node[style=vertex](11) at (11,31) {};
\draw node[style=vertex](12) at (12,32) {};
\draw node[style=vertex](13) at (13,22) {};
\draw node[style=vertex](14) at (14,6) {};
\draw node[style=vertex](15) at (15,9) {};
\draw node[style=vertex](16) at (16,23) {};
\draw node[style=vertex](17) at (17,12) {};
\draw node[style=vertex](18) at (18,10) {};
\draw node[style=vertex](19) at (19,25) {};
\draw node[style=vertex](20) at (20,27) {};
\draw node[style=vertex](21) at (21,24) {};
\draw node[style=vertex](22) at (22,13) {};
\draw node[style=vertex](23) at (23,14) {};
\draw node[style=vertex](24) at (24,18) {};
\draw node[style=vertex](25) at (25,15) {};
\draw node[style=vertex](26) at (26,11) {};
\draw node[style=vertex](27) at (27,28) {};
\draw node[style=vertex](28) at (28,29) {};
\draw node[style=vertex](29) at (29,26) {};
\draw node[style=vertex](30) at (30,17) {};
\draw node[style=vertex](31) at (31,19) {};
\draw node[style=vertex](32) at (32,33) {};
\draw node[style=vertex](33) at (33,16) {};
\draw (0) -- (1) -- (2) -- (3) -- (4) -- (5) -- (6) -- (7) -- (8) -- (9) -- (10) -- (11) -- (12) -- (13) -- (14) -- (15) -- (16) -- (17) -- (18) -- (19) -- (20) -- (21) -- (22) -- (23) -- (24) -- (25) -- (26) -- (27) -- (28) -- (29) -- (30) -- (31) -- (32) -- (33) ;
\end{tikzpicture}
\caption{}
\end{subfigure}
\begin{subfigure}{\textwidth}
\begin{tikzpicture}[xscale=\textwidth/36cm, yscale=4cm/36cm]
\draw node[style=vertex](0) at (0,8) {};
\draw node[style=vertex](1) at (1,2) {};
\draw node[style=vertex](2) at (2,0) {};
\draw node[style=vertex](3) at (3,13) {};
\draw node[style=vertex](4) at (4,14) {};
\draw node[style=vertex](5) at (5,9) {};
\draw node[style=vertex](6) at (6,3) {};
\draw node[style=vertex](7) at (7,4) {};
\draw node[style=vertex](8) at (8,7) {};
\draw node[style=vertex](9) at (9,5) {};
\draw node[style=vertex](10) at (10,1) {};
\draw node[style=vertex](11) at (11,16) {};
\draw node[style=vertex](12) at (12,23) {};
\draw node[style=vertex](13) at (13,15) {};
\draw node[style=vertex](14) at (14,6) {};
\draw node[style=vertex](15) at (15,10) {};
\draw node[style=vertex](16) at (16,24) {};
\draw node[style=vertex](17) at (17,17) {};
\draw node[style=vertex](18) at (18,11) {};
\draw node[style=vertex](19) at (19,26) {};
\draw node[style=vertex](20) at (20,28) {};
\draw node[style=vertex](21) at (21,25) {};
\draw node[style=vertex](22) at (22,18) {};
\draw node[style=vertex](23) at (23,19) {};
\draw node[style=vertex](24) at (24,22) {};
\draw node[style=vertex](25) at (25,20) {};
\draw node[style=vertex](26) at (26,12) {};
\draw node[style=vertex](27) at (27,31) {};
\draw node[style=vertex](28) at (28,33) {};
\draw node[style=vertex](29) at (29,30) {};
\draw node[style=vertex](30) at (30,21) {};
\draw node[style=vertex](31) at (31,27) {};
\draw node[style=vertex](32) at (32,34) {};
\draw node[style=vertex](33) at (33,32) {};
\draw node[style=vertex](34) at (34,29) {};
\draw node[style=vertex](35) at (35,35) {};
\draw (0) -- (1) -- (2) -- (3) -- (4) -- (5) -- (6) -- (7) -- (8) -- (9) -- (10) -- (11) -- (12) -- (13) -- (14) -- (15) -- (16) -- (17) -- (18) -- (19) -- (20) -- (21) -- (22) -- (23) -- (24) -- (25) -- (26) -- (27) -- (28) -- (29) -- (30) -- (31) -- (32) -- (33) -- (34) -- (35) ;
\end{tikzpicture}
\caption{}
\end{subfigure}
\begin{subfigure}{\textwidth}
\begin{tikzpicture}[xscale=\textwidth/40cm, yscale=4cm/40cm]
\draw node[style=vertex](0) at (0,8) {};
\draw node[style=vertex](1) at (1,2) {};
\draw node[style=vertex](2) at (2,0) {};
\draw node[style=vertex](3) at (3,20) {};
\draw node[style=vertex](4) at (4,29) {};
\draw node[style=vertex](5) at (5,19) {};
\draw node[style=vertex](6) at (6,3) {};
\draw node[style=vertex](7) at (7,4) {};
\draw node[style=vertex](8) at (8,7) {};
\draw node[style=vertex](9) at (9,5) {};
\draw node[style=vertex](10) at (10,1) {};
\draw node[style=vertex](11) at (11,30) {};
\draw node[style=vertex](12) at (12,31) {};
\draw node[style=vertex](13) at (13,21) {};
\draw node[style=vertex](14) at (14,6) {};
\draw node[style=vertex](15) at (15,9) {};
\draw node[style=vertex](16) at (16,22) {};
\draw node[style=vertex](17) at (17,12) {};
\draw node[style=vertex](18) at (18,10) {};
\draw node[style=vertex](19) at (19,24) {};
\draw node[style=vertex](20) at (20,26) {};
\draw node[style=vertex](21) at (21,23) {};
\draw node[style=vertex](22) at (22,13) {};
\draw node[style=vertex](23) at (23,14) {};
\draw node[style=vertex](24) at (24,17) {};
\draw node[style=vertex](25) at (25,15) {};
\draw node[style=vertex](26) at (26,11) {};
\draw node[style=vertex](27) at (27,27) {};
\draw node[style=vertex](28) at (28,28) {};
\draw node[style=vertex](29) at (29,25) {};
\draw node[style=vertex](30) at (30,16) {};
\draw node[style=vertex](31) at (31,18) {};
\draw node[style=vertex](32) at (32,34) {};
\draw node[style=vertex](33) at (33,33) {};
\draw node[style=vertex](34) at (34,32) {};
\draw node[style=vertex](35) at (35,37) {};
\draw node[style=vertex](36) at (36,39) {};
\draw node[style=vertex](37) at (37,36) {};
\draw node[style=vertex](38) at (38,35) {};
\draw node[style=vertex](39) at (39,38) {};
\draw (0) -- (1) -- (2) -- (3) -- (4) -- (5) -- (6) -- (7) -- (8) -- (9) -- (10) -- (11) -- (12) -- (13) -- (14) -- (15) -- (16) -- (17) -- (18) -- (19) -- (20) -- (21) -- (22) -- (23) -- (24) -- (25) -- (26) -- (27) -- (28) -- (29) -- (30) -- (31) -- (32) -- (33) -- (34) -- (35) -- (36) -- (37) -- (38) -- (39) ;
\end{tikzpicture}
\caption{}
\end{subfigure}
\end{figure}
\begin{figure}\ContinuedFloat
\begin{subfigure}{\textwidth}
\begin{tikzpicture}[xscale=\textwidth/42cm, yscale=4cm/42cm]
\draw node[style=vertex](0) at (0,8) {};
\draw node[style=vertex](1) at (1,2) {};
\draw node[style=vertex](2) at (2,0) {};
\draw node[style=vertex](3) at (3,20) {};
\draw node[style=vertex](4) at (4,29) {};
\draw node[style=vertex](5) at (5,19) {};
\draw node[style=vertex](6) at (6,3) {};
\draw node[style=vertex](7) at (7,4) {};
\draw node[style=vertex](8) at (8,7) {};
\draw node[style=vertex](9) at (9,5) {};
\draw node[style=vertex](10) at (10,1) {};
\draw node[style=vertex](11) at (11,30) {};
\draw node[style=vertex](12) at (12,31) {};
\draw node[style=vertex](13) at (13,21) {};
\draw node[style=vertex](14) at (14,6) {};
\draw node[style=vertex](15) at (15,9) {};
\draw node[style=vertex](16) at (16,22) {};
\draw node[style=vertex](17) at (17,12) {};
\draw node[style=vertex](18) at (18,10) {};
\draw node[style=vertex](19) at (19,24) {};
\draw node[style=vertex](20) at (20,26) {};
\draw node[style=vertex](21) at (21,23) {};
\draw node[style=vertex](22) at (22,13) {};
\draw node[style=vertex](23) at (23,14) {};
\draw node[style=vertex](24) at (24,17) {};
\draw node[style=vertex](25) at (25,15) {};
\draw node[style=vertex](26) at (26,11) {};
\draw node[style=vertex](27) at (27,27) {};
\draw node[style=vertex](28) at (28,28) {};
\draw node[style=vertex](29) at (29,25) {};
\draw node[style=vertex](30) at (30,16) {};
\draw node[style=vertex](31) at (31,18) {};
\draw node[style=vertex](32) at (32,34) {};
\draw node[style=vertex](33) at (33,33) {};
\draw node[style=vertex](34) at (34,32) {};
\draw node[style=vertex](35) at (35,39) {};
\draw node[style=vertex](36) at (36,41) {};
\draw node[style=vertex](37) at (37,38) {};
\draw node[style=vertex](38) at (38,36) {};
\draw node[style=vertex](39) at (39,37) {};
\draw node[style=vertex](40) at (40,40) {};
\draw node[style=vertex](41) at (41,35) {};
\draw (0) -- (1) -- (2) -- (3) -- (4) -- (5) -- (6) -- (7) -- (8) -- (9) -- (10) -- (11) -- (12) -- (13) -- (14) -- (15) -- (16) -- (17) -- (18) -- (19) -- (20) -- (21) -- (22) -- (23) -- (24) -- (25) -- (26) -- (27) -- (28) -- (29) -- (30) -- (31) -- (32) -- (33) -- (34) -- (35) -- (36) -- (37) -- (38) -- (39) -- (40) -- (41) ;
\end{tikzpicture}
\caption{}
\end{subfigure}

\begin{subfigure}{\textwidth}
\begin{tikzpicture}[xscale=\textwidth/44cm, yscale=4cm/44cm]
\draw node[style=vertex](0) at (0,8) {};
\draw node[style=vertex](1) at (1,2) {};
\draw node[style=vertex](2) at (2,0) {};
\draw node[style=vertex](3) at (3,20) {};
\draw node[style=vertex](4) at (4,29) {};
\draw node[style=vertex](5) at (5,19) {};
\draw node[style=vertex](6) at (6,3) {};
\draw node[style=vertex](7) at (7,4) {};
\draw node[style=vertex](8) at (8,7) {};
\draw node[style=vertex](9) at (9,5) {};
\draw node[style=vertex](10) at (10,1) {};
\draw node[style=vertex](11) at (11,30) {};
\draw node[style=vertex](12) at (12,31) {};
\draw node[style=vertex](13) at (13,21) {};
\draw node[style=vertex](14) at (14,6) {};
\draw node[style=vertex](15) at (15,9) {};
\draw node[style=vertex](16) at (16,22) {};
\draw node[style=vertex](17) at (17,12) {};
\draw node[style=vertex](18) at (18,10) {};
\draw node[style=vertex](19) at (19,24) {};
\draw node[style=vertex](20) at (20,26) {};
\draw node[style=vertex](21) at (21,23) {};
\draw node[style=vertex](22) at (22,13) {};
\draw node[style=vertex](23) at (23,14) {};
\draw node[style=vertex](24) at (24,17) {};
\draw node[style=vertex](25) at (25,15) {};
\draw node[style=vertex](26) at (26,11) {};
\draw node[style=vertex](27) at (27,27) {};
\draw node[style=vertex](28) at (28,28) {};
\draw node[style=vertex](29) at (29,25) {};
\draw node[style=vertex](30) at (30,16) {};
\draw node[style=vertex](31) at (31,18) {};
\draw node[style=vertex](32) at (32,34) {};
\draw node[style=vertex](33) at (33,33) {};
\draw node[style=vertex](34) at (34,32) {};
\draw node[style=vertex](35) at (35,38) {};
\draw node[style=vertex](36) at (36,42) {};
\draw node[style=vertex](37) at (37,37) {};
\draw node[style=vertex](38) at (38,35) {};
\draw node[style=vertex](39) at (39,36) {};
\draw node[style=vertex](40) at (40,41) {};
\draw node[style=vertex](41) at (41,40) {};
\draw node[style=vertex](42) at (42,39) {};
\draw node[style=vertex](43) at (43,43) {};
\draw (0) -- (1) -- (2) -- (3) -- (4) -- (5) -- (6) -- (7) -- (8) -- (9) -- (10) -- (11) -- (12) -- (13) -- (14) -- (15) -- (16) -- (17) -- (18) -- (19) -- (20) -- (21) -- (22) -- (23) -- (24) -- (25) -- (26) -- (27) -- (28) -- (29) -- (30) -- (31) -- (32) -- (33) -- (34) -- (35) -- (36) -- (37) -- (38) -- (39) -- (40) -- (41) -- (42) -- (43) ;
\end{tikzpicture}
\caption{}
\end{subfigure}

\begin{subfigure}{\textwidth}
\begin{tikzpicture}[xscale=\textwidth/46cm, yscale=4cm/46cm]
\draw node[style=vertex](0) at (0,8) {};
\draw node[style=vertex](1) at (1,2) {};
\draw node[style=vertex](2) at (2,0) {};
\draw node[style=vertex](3) at (3,25) {};
\draw node[style=vertex](4) at (4,34) {};
\draw node[style=vertex](5) at (5,24) {};
\draw node[style=vertex](6) at (6,3) {};
\draw node[style=vertex](7) at (7,4) {};
\draw node[style=vertex](8) at (8,7) {};
\draw node[style=vertex](9) at (9,5) {};
\draw node[style=vertex](10) at (10,1) {};
\draw node[style=vertex](11) at (11,35) {};
\draw node[style=vertex](12) at (12,36) {};
\draw node[style=vertex](13) at (13,26) {};
\draw node[style=vertex](14) at (14,6) {};
\draw node[style=vertex](15) at (15,9) {};
\draw node[style=vertex](16) at (16,27) {};
\draw node[style=vertex](17) at (17,12) {};
\draw node[style=vertex](18) at (18,10) {};
\draw node[style=vertex](19) at (19,29) {};
\draw node[style=vertex](20) at (20,31) {};
\draw node[style=vertex](21) at (21,28) {};
\draw node[style=vertex](22) at (22,13) {};
\draw node[style=vertex](23) at (23,14) {};
\draw node[style=vertex](24) at (24,17) {};
\draw node[style=vertex](25) at (25,15) {};
\draw node[style=vertex](26) at (26,11) {};
\draw node[style=vertex](27) at (27,32) {};
\draw node[style=vertex](28) at (28,33) {};
\draw node[style=vertex](29) at (29,30) {};
\draw node[style=vertex](30) at (30,16) {};
\draw node[style=vertex](31) at (31,23) {};
\draw node[style=vertex](32) at (32,39) {};
\draw node[style=vertex](33) at (33,38) {};
\draw node[style=vertex](34) at (34,37) {};
\draw node[style=vertex](35) at (35,40) {};
\draw node[style=vertex](36) at (36,41) {};
\draw node[style=vertex](37) at (37,22) {};
\draw node[style=vertex](38) at (38,19) {};
\draw node[style=vertex](39) at (39,20) {};
\draw node[style=vertex](40) at (40,43) {};
\draw node[style=vertex](41) at (41,42) {};
\draw node[style=vertex](42) at (42,21) {};
\draw node[style=vertex](43) at (43,44) {};
\draw node[style=vertex](44) at (44,45) {};
\draw node[style=vertex](45) at (45,18) {};
\draw (0) -- (1) -- (2) -- (3) -- (4) -- (5) -- (6) -- (7) -- (8) -- (9) -- (10) -- (11) -- (12) -- (13) -- (14) -- (15) -- (16) -- (17) -- (18) -- (19) -- (20) -- (21) -- (22) -- (23) -- (24) -- (25) -- (26) -- (27) -- (28) -- (29) -- (30) -- (31) -- (32) -- (33) -- (34) -- (35) -- (36) -- (37) -- (38) -- (39) -- (40) -- (41) -- (42) -- (43) -- (44) -- (45) ;
\end{tikzpicture}
\caption{}
\end{subfigure}
\caption{Bicrucial permutations of lengths 34, 36, 40, 42, 44 and 46 respectively.}
\label{fig:small-cases}
\end{figure}
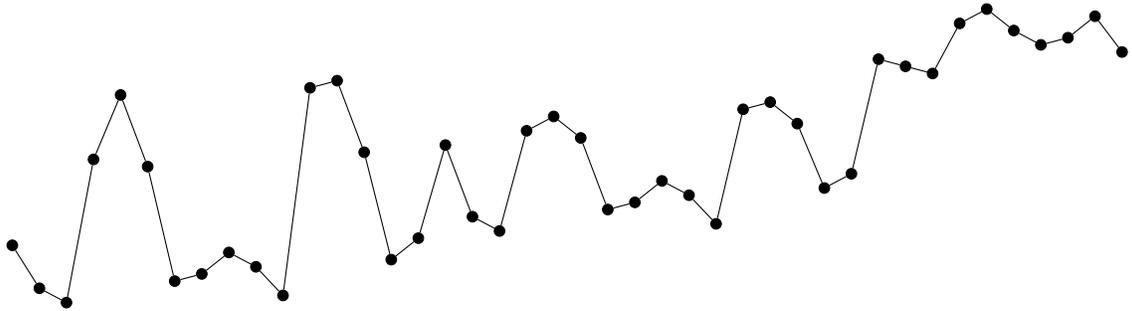
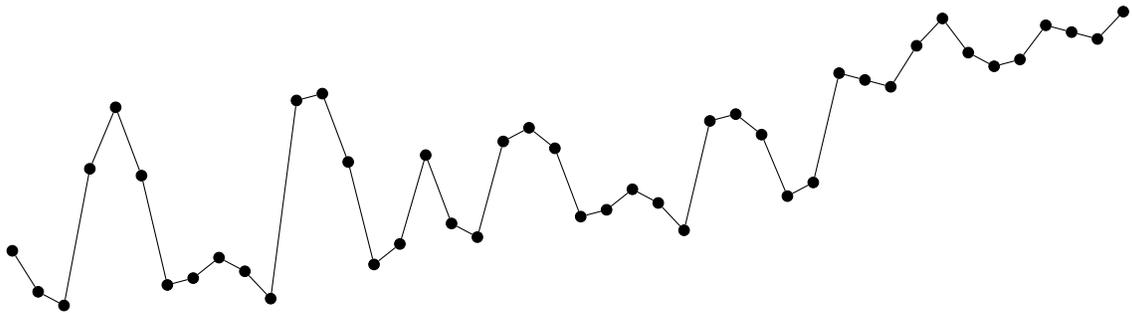
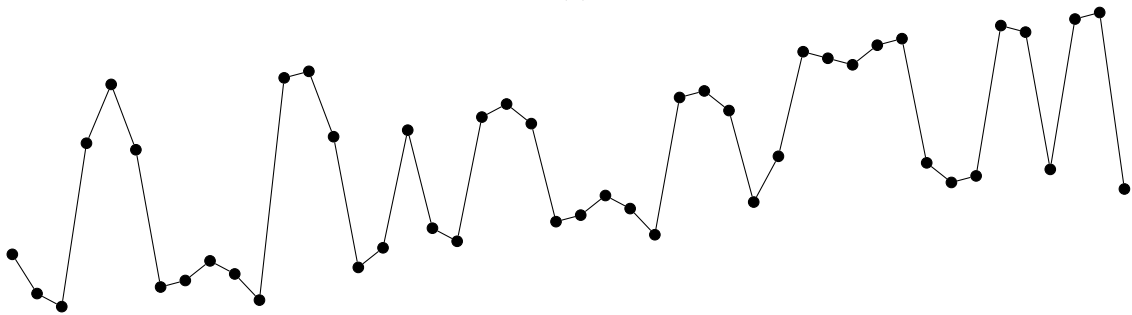

%% file: main.bbl
\begin{thebibliography}{10}

\bibitem{avgustinovich2011square}
S.~Avgustinovich, S.~Kitaev, A.~Pyatkin, and A.~Valyuzhenich.
\newblock On square-free permutations.
\newblock {\em Journal of Automata, Languages and Combinatorics}, 16(1):3--10,
  2011.

\bibitem{bean1979avoidable}
D.~Bean, A.~Ehrenfeucht, and G.~McNulty.
\newblock Avoidable patterns in strings of symbols.
\newblock {\em Pacific Journal of Mathematics}, 85(2):261--294, 1979.

\bibitem{gent2015crucial}
I.~Gent, S.~Kitaev, A.~Konovalov, S.~Linton, and P.~Nightingale.
\newblock {$S$-crucial and bicrucial permutations with respect to squares}.
\newblock {\em Journal of Integer Sequences}, 18(6), 2015.

\bibitem{grytczuk2020extremal}
J.~Grytczuk, H.~Kordulewski, and A.~Niewiadomski.
\newblock Extremal square-free words.
\newblock {\em Electronic Journal of Combinatorics}, 27(1), 2020.

\bibitem{hong2021no}
L.~Hong and S.~Zhang.
\newblock No extremal square-free words over large alphabets.
\newblock {\em arXiv preprint arXiv:2107.13123}, 2021.

\bibitem{Kitaev}
S.~Kitaev.
\newblock Personal communication at \textit{Permutation Patterns 2021}.

\bibitem{li1976annihilators}
S.-Y.~R. Li.
\newblock Annihilators in nonrepetitive semigroups.
\newblock {\em Studies in Applied Mathematics}, 55(1):83--85, 1976.

\bibitem{mol2020lengths}
L.~Mol and N.~Rampersad.
\newblock Lengths of extremal square-free ternary words.
\newblock {\em Contributions to Discrete Mathematics}, to appear.

\bibitem{oeis}
{The OEIS Foundation Inc.}
\newblock The {O}n-{L}ine {E}ncyclopedia of {I}nteger {S}equences, 2021.

\bibitem{thue1906uber}
A.~Thue.
\newblock Uber unendliche zeichenreihen.
\newblock {\em Skrifter udgivne af Videnskabsselskabet i Christiania. I
  Mathematisk-naturvidenskabelig klasse.}, 7:1--22, 1906.

\end{thebibliography}
